\documentclass[preprint,none]{imsart}
\RequirePackage[OT1]{fontenc}
\RequirePackage{amsthm,amsmath,amssymb,natbib,multirow,booktabs}
\usepackage{color}
\usepackage{graphicx}

\newcommand{\mc}{\mathcal}
\newcommand{\mb}{\mathbb}
\newcommand{\E}{\mathbb{E}}

\newcommand{\ti}{\widetilde}
\newcommand{\ol}{\overline}

\newcommand{\eps}{\varepsilon}
\newcommand{\lo}{\langle}
\newcommand{\ro}{\rangle}
\newcommand{\R}{\mathbb{R}}
\newcommand{\tr}{{\text{\tiny{\textsf{T}}}}}

\DeclareMathOperator*\sgn{sgn}
\renewcommand{\phi}{\varphi}
\renewcommand{\epsilon}{\eps}

\theoremstyle{definition}
\newtheorem{defn}{Definition}[section]
\newtheorem{rmk}[defn]{Remark}

\theoremstyle{plain}
\newtheorem{lem}[defn]{Lemma}
\newtheorem{thm}[defn]{Theorem}
\newtheorem{cor}[defn]{Corollary}

\newtheorem{ass}[defn]{Assumption}
\newtheorem{prop}[defn]{Proposition}
\numberwithin{equation}{section}

\begin{document}
\begin{frontmatter}

% "Title of the paper"
\title{\ \ \ \ \ \ \ \ \ \ BSDEs in Utility Maximization with\ \ \ \ \ \ \ \ \ \  \hbox{} BMO Market Price of Risk}
\runtitle{BSDEs with BMO market price of risk}

\author{\fnms{Christoph} \snm{Frei\thanksref{t1},} \ead[label=e1]{cfrei@ualberta.ca}}
\address{Christoph Frei\\ Mathematical and Statistical Sciences\\
University of Alberta \\ Edmonton AB T6G 2G1\\ Canada\\
\printead{e1}}
\author{\fnms{\hspace{-0.25cm}Markus} \snm{Mocha\thanksref{t2}\hspace{-0.175cm}
}\ead[label=e2]{mocha@math.hu-berlin.de}}
\address{Markus Mocha\\ Institut f\"{u}r Mathematik\\
Humboldt-Universit\"{a}t zu Berlin \\ Unter den Linden 6, 10099 Berlin \\ Germany\\
\printead{e2}}
\and
\author{\fnms{Nicholas}  \snm{Westray}\corref{}\ead[label=e4]{n.westray@imperial.ac.uk}}
\address{Nicholas Westray\\ Department of Mathematics\\ Imperial College \\ London SW7 2AZ \\ United Kingdom\\
\printead{e4}}
\affiliation{University of Alberta, Humboldt-Universit\"{a}t zu Berlin and Imperial College London}
\thankstext{t1}{C. Frei gratefully acknowledges financial support by the Natural Sciences and Engineering Research Council of Canada through
grant 402585.}
\thankstext{t2}{Corresponding author: \texttt{mocha@math.hu-berlin.de} \hspace{15cm}  a
\quad We thank two anonymous referees for very helpful comments, which enabled us to improve the~paper. }
\runauthor{C. Frei, M. Mocha and N. Westray}



\begin{abstract}
This article studies quadratic semimartingale BSDEs arising in power utility maximization when the market price of risk is of BMO type. In a Brownian setting we provide a necessary and sufficient condition for the existence of a solution but show that uniqueness fails to hold in the sense that there exists a continuum of distinct square-integrable solutions. This feature occurs since, contrary to the classical It\^{o} representation theorem, a representation of random variables in terms of stochastic exponentials is not unique. We study in detail when the BSDE has a bounded solution and derive a new dynamic exponential moments condition which is shown to be the minimal sufficient condition in a general filtration. The main results are complemented by several interesting examples which illustrate their sharpness as well as important properties of the utility maximization BSDE.\vspace{-2pt}
\end{abstract}
\end{frontmatter}

%%%%%%%%%%%%%%%%%%%%%%%%%%%%%%%%%%%%%%%%%%%%%%%%%%
%Introduction for utility valued on half line
%%%%%%%%%%%%%%%%%%%%%%%%%%%%%%%%%%%%%%%%%%%%%%%%%%

\section{Introduction}
In this article we study quadratic semimartingale BSDEs that arise in power utility maximization. More precisely, we consider the optimal investment problem over a finite time horizon $[0,T]$ for an agent whose goal is to maximize the expected utility of terminal wealth. Such a problem is classical in mathematical finance and dates back to Merton \cite{Me69}. For general utility functions (not necessarily power) the main solution technique 
is convex duality, see Karatzas and Shreve \cite{KS98}, Kramkov and Schachermayer \cite{KS99} as well as the survey article of Schachermayer \cite{S04} which gives an excellent overview of the ideas involved as well as many further references.

A second approach to tackling the above problem is via BSDEs, using the factorization property of the value process when the utility function is of power type. This allows one to apply the martingale optimality principle and, as shown in Hu, Imkeller and M\"{u}ller \cite{HIM05}, to describe the value process and optimal trading strategy completely via a BSDE. Their article can be regarded as an extension of earlier work by Rouge and El Karoui \cite{REK00} as well as Sekine \cite{Se06} and relies on existence and uniqueness results for quadratic BSDEs first proved in Kobylanski \cite{K00}. Those existence results were subsequently extended by Morlais \cite{Mo09}, Briand and Hu \cite{BH06, BH08}, Delbaen, Hu and Richou \cite{DHR09} as well as Barrieu and El Karoui \cite{BEK11}. In particular, the assumption that the mean-variance tradeoff process be bounded has been weakened to it having certain exponential moments, see Mocha and Westray \cite{MW10,MW10a}.

As previously stated, it is the BSDE approach to tackling the utility maximization problem which is studied in detail in the present paper. The key idea is to derive a BSDE whose solution provides a candidate optimal wealth process together with a candidate optimal strategy. Then a verification argument is applied, showing that these are in fact optimal, see Nutz \cite{Nu209}. This latter step is the difficult part and typically requires extra regularity of the BSDE solution which is guaranteed by the boundedness or existence of exponential moments of the mean-variance tradeoff. In particular, when one can show the existence of a bounded solution, verification is feasible. This is due to the fact that the martingale part of the corresponding BSDE is then known to be a BMO martingale. Motivated by the ease of verification given a bounded solution the main aim of this paper is to quantify, in terms of assumptions on the mean-variance tradeoff process, when one can expect a bounded solution. This natural question justifies the present study.

When the market is continuous (as assumed in this article) one can write the mean-variance tradeoff as $\lo \lambda \cdot M, \lambda \cdot M\ro_T$ for a continuous local martingale $M$ and predictable $M$-integrable process $\lambda$. The assumption that the mean-variance 
tradeoff process be bounded or have all exponential moments implies that the minimal martingale measure $\mc{E}(-\lambda\cdot M)_T$ is a true probability measure. In particular, the set of equivalent martingale measures is nonempty, so that there is no arbitrage in the sense of NFLVR, see Delbaen and Schachermayer \cite{DS98}. If the local martingale $\lambda\cdot M$ is instead assumed to be only a BMO martingale then from Kazamaki \cite{Ka94} the minimal martingale measure is again a true probability measure and NFLVR holds. In this case the mean-variance tradeoff now need not be bounded or have all exponential moments. A secondary objective of this paper is to study what happens to the solution of the BSDE in this situation. As discussed, such a condition on $\lambda\cdot M$ arises naturally from a no-arbitrage point of view, additionally however there is a known relation between boundedness of solutions to quadratic BSDEs and BMO martingales so that this question is also interesting from a mathematical standpoint.

The present article has three main contributions. Firstly in a Brownian setting, 
we show that the BSDE admits a continuum of distinct solutions with square-integrable martingale parts. This result provides square-integrable counterexamples to uniqueness of BSDE solutions. The spirit of our counterexamples is similar to that of Ankirchner, Imkeller and Popier \cite{AIP08} Section 2.2 except that we consider BSDEs related to the utility maximization problem. Our result stems from the fact that contrary to the classical It\^{o} representation formula with square-integrable integrands, a ``multiplicative'' $L^2$-analogue in terms of stochastic exponentials is not unique, see Lemma \ref{itoexplem}.

The second contribution is a thorough investigation of when the BSDE admits a bounded solution. If the investor's relative risk aversion is greater than one and $\lambda\cdot M$ is a BMO martingale, this is automatically satisfied. For a risk aversion smaller than one the picture is rather different and we provide an example to show that even when the mean-variance tradeoff has all exponential moments and the process $\lambda\cdot M$ is a BMO martingale, the solution to the utility maximization BSDE need not be bounded. Building on this example our third and most important result is Theorem \ref{dem}, which shows how to combine the BMO and exponential moment conditions so as to find the minimal condition which guarantees, in a general filtration, that the BSDE admits a bounded solution. We thus fully characterize the boundedness of solutions to the quadratic BSDE arising in power utility maximization. We mention that the limiting case of risk aversion equal to one, i.e. the case of logarithmic utility, is covered by our results.

The paper is organized as follows. In the next section we establish the link between the utility maximization problem and BSDEs for an unbounded mean-variance tradeoff. Then we analyze the questions of existence and uniqueness of BSDE solutions and in Section \ref{sec5} turn our attention to the interplay between boundedness of solutions and the BMO property of $\lambda\cdot M$. In Section \ref{SecCountBound} we develop some related counterexamples and then provide the characterization of boundedness in Section \ref{sec6}. In an additional appendix we collect some background material on quadratic continuous semimartingale BSDEs.

%%%%%%%%%%%%%%%%%%%%%%%%%%%%%%%%%%%%%%%%%%%%%%%%%%%%%%%%%%%%%
%BSDEs in Power Utility Maximization
%%%%%%%%%%%%%%%%%%%%%%%%%%%%%%%%%%%%%%%%%%%%%%%%%%%%%%%%%%%%%

\section{Power Utility Maximization and Quadratic BSDEs}\label{SecRelation}
Throughout this article we work on a filtered probability space $(\Omega,\mc{F},(\mc{F}_t)_{0\leq t\leq T},\mb{P})$ satisfying the usual conditions of right-continuity and completeness. We assume that the time horizon $T$ is a finite number in $(0,\infty)$ and that $\mc{F}_0$ is the completion of the trivial $\sigma$-algebra. In addition all semimartingales are assumed to be equal to their c\`{a}dl\`{a}g modification. In a first step, we assume that the filtration is \emph{continuous} in the sense that all local martingales are continuous. This condition is relaxed in Section \ref{sec6} where we provide the main characterization result for a general filtration.

There is a market consisting of one bond, assumed constant, and $d$ stocks with price process $S=(S^1,\ldots,S^d)^\tr\!$, which we assume to have dynamics
\begin{equation*}
dS_t=\text{Diag}(S_t)\,\bigl(dM_t+ \,d\lo M,M \ro_t\lambda_t\bigr),
\end{equation*}
where $M=(M^1,\ldots,M^d)^\tr$ is a $d$-dimensional continuous local martingale with $M_0=0$, $\lambda$ is a $d$-dimensional predictable process, the \emph{market price of risk}, satisfying
\begin{equation*}
\int_0^T \lambda_t^\tr \,d\lo M,M \ro_t\lambda_t <+\infty\quad\text{ a.s.}
\end{equation*}
and $\text{Diag}(S)$ denotes the $d\times d$ diagonal matrix whose elements are taken from $S$.

We consider an investor trading in the above market according to an admissible investment strategy $\nu$. A predictable $d$-dimensional process $\nu$ is called \emph{admissible} if it is $M$-integrable, i.e. \mbox{$\int_0^T\nu^\tr_t \,d\lo M,M \ro_t\nu_t <+\infty$} a.s. and we write $\mc{A}$ for the family of such investment strategies $\nu$. Observe that each component $\nu^i$ represents the \emph{proportion of wealth} invested in the $i$th stock $S^i$, $i=1,\ldots,d$. In particular, for some initial capital $x>0$ and an admissible strategy $\nu$, the associated \emph{wealth process} $X^{x,\nu}$ evolves as follows,
\begin{equation*}
X^{x,\nu}:=x\,\mc{E}(\nu\cdot M +\nu\cdot \lo M,M\ro\lambda), 
\end{equation*}
where $\mc{E}$ denotes the stochastic exponential. The family of all such wealth processes is denoted by $\mc{X}(x)$.

Our investor has preferences modelled by a power utility function $U$, \[U(x)=\frac{\,x^p}{p},\quad\text{ where }p\in(-\infty,0)\cup(0,1).\]
Starting with initial capital $x>0$, they aim to maximize the expected utility of terminal wealth. This leads to the following primal optimization problem, 
\begin{equation}\label{primalproblem}
u(x):=\sup_{\nu\in\mc{A}}\,\E\!\left[U\bigl(X^{x,\nu}_T\bigr)\right].
\end{equation}
Related to the above primal problem is a dual problem which we now describe. For $y>0$ we introduce the set
\begin{equation*}
\mc{Y}(y):=\left\{Y\geq 0\,|\,Y_0=y \text{ and } XY \text{
is a supermartingale for all }X\in\mc{X}(1)\right\}\!,
\end{equation*}
as well as the minimization problem
\begin{equation*} 
\ti{u}(y):=\inf_{Y\in\mc{Y}(y)}\E\!\left[\ti{U}\big(Y_T\big)\right],
\end{equation*}
where $\ti{U}$ is the \emph{conjugate} (or \emph{dual}) of $U$ given by $\ti{U}(y)=-\frac{\,y^q}{q}$, $y>0$, where $q:=\frac{p}{p-1}$ is the \emph{dual exponent} to $p$. There is a bijection between $p$ and $q$ so that in what follows we often state the results for $q$ rather than for $p$. 

It is shown in Kramkov and Schachermayer \cite{KS99,KS03} (among others) that for general utility functions (not necessarily power) the following assumption is the weakest possible for well posedness of the market model and the utility maximization problem.
\begin{ass}\label{ass2}\mbox{}
\begin{enumerate}
\item The set $\mc{M}^e(S)$ of equivalent local martingale measures for $S$ is non-empty.
\item If $p>0$, there is an $x>0$ such that $u(x)<+\infty.$
\end{enumerate}
\end{ass}
Summarizing the results of \cite{KS99} we then have that there exists a strategy $\hat{\nu}\in\,\mc{A}$, independent of $x>0$, which is \emph{optimal} for the primal problem. In particular, for any other optimal strategy $\bar{\nu}\in\,\mc{A}$ and $x>0$ we have that $\hat{X}^x:=X^{x,\hat{\nu}}$ and $X^{x,\bar{\nu}}$ are indistinguishable. On the dual side, given $y>0$, there exists a $\hat{Y}^y\in\mc{Y}(y)$ which is optimal for the dual problem and unique up to indistinguishability. Finally, the functions $u$ and $\ti{u}$ are continuously differentiable and conjugate and if $y=u'(x)$ then $\hat{Y}_T = U'(\hat{X}_T)$ and the process $\hat{X}\hat{Y}$ is a martingale on $[0,T]$, where we omit writing the dependence on the initial values.
\begin{rmk}
In \cite{KS99} the authors work in the additive formulation where strategies represent the number of shares and wealth remains (only) nonnegative. However, for power utility maximization, the additive formulation and the setting here are equivalent. We refer the reader to \cite{MW10a} for further details.
\end{rmk}
We also state here a representation of the dual optimizer shown in \cite{LZ07} when $S$ is one-dimensional and extended in \cite{MW10a} to the current framework. Using the continuity of the filtration, there exists a continuous local martingale $\hat{N}$ which is orthogonal to $M$, i.e. $\lo\hat{N},M^i\ro\equiv 0$ for all $i=1,\ldots,d$, and such that $\hat{Y}=y\,\mc{E}(-\lambda\cdot M + \hat{N})$.

Finally, we recall the BSDE satisfied by the so-called \emph{opportunity process} from \cite{Nu109}, more precisely by its log-transform. We start with the solutions $\hat{X}$ and $\hat{Y}$ to the above primal and dual problem (when $y=u'(x)$) and derive the BSDE satisfied by the following process
\begin{equation*}
\hat{\Psi}:=\log\!\left(\frac{\hat{Y}}{U'(\hat{X})}\right).
\end{equation*}
The logic is now very similar to the procedure in \cite{MS05}, we apriori obtained the existence of the object $\hat{\Psi}$ of interest. Imposing a suitable assumption we show that $\hat{\Psi}$ lies in a certain space in which solutions to (a special type of) quadratic semimartingale BSDE
are unique. This approach of using BSDE comparison principles in utility maximization may be found in \cite{HIM05} and \cite{Mo09}. We observe that in these references the \emph{mean-variance tradeoff} $\lo \lambda\cdot M,\lambda\cdot M\ro_T$ is bounded. In what follows we extend their reasoning to the unbounded case under exponential moments which was thoroughly investigated in \cite{MW10a} where the corresponding assumption is
\begin{ass}\label{ass_expmom}
For all $\varrho>0$ we have that 
\[\E\Bigl[\exp\bigl(\varrho\,\lo \lambda\cdot M,\lambda\cdot M\ro_T\bigr)\Bigr]<+\infty,\]
i.e., the mean-variance tradeoff $\lo \lambda\cdot M,\lambda\cdot M\ro_T$ has exponential moments of all orders.
\end{ass}
The preceding assumption is compatible with Assumption \ref{ass2} in the following sense.
\begin{lem}[\cite{MW10a} Proposition 2.8]\label{Ass2341}
Assumption \ref{ass_expmom} implies Assumption \ref{ass2}, more precisely,
\begin{enumerate}
\item The process $Y^\lambda:=\mc{E}(-\lambda\cdot M)$ is a martingale on
$[0,T]$, hence defines an equivalent local martingale measure for
$S$. \item The function $u$ is finite on all of $(0,+\infty)$.
\end{enumerate}
\end{lem}
For completeness let us also include the case $p=0$, equivalently $q=0$, which corresponds to the case of logarithmic utility, i.e. $U(x)=\log(x)$ and $\ti{U}(y)=-\log(y)-1$, and for which the optimizers  
have a simple structure. Namely, $\hat{\nu}\equiv\lambda$ is the optimal strategy and $\hat{Y}^1\equiv\mc{E}(-\lambda\cdot M)$ is the dual optimizer. We also have $\hat{X}\equiv\hat{Y}^{-1}$ so that $\hat{\Psi}\equiv0$. We mention that the Assumption \ref{ass2} with item (ii) extended to $p=0$ is a sufficient condition for existence and uniqueness of the optimizers. In particular, Assumption \ref{ass_expmom} or the condition that $\lambda\cdot M$ be a BMO martingale, see below, are sufficient as can be deduced easily.\\ 

Before we discuss properties of the process $\hat{\Psi}$ we first fix some notation.
\begin{defn}
Let $\mathfrak{E}$ denote the space of all processes $\Upsilon$ on $[0,T]$ whose supremum $\Upsilon^*:=\sup_{0\leq t\leq T}|\Upsilon_t|$ has exponential moments of all orders, i.e. for all $\varrho>0$,
\[\E\left[\exp\left(\varrho\Upsilon^*\right)\right]<+\infty.\]
\end{defn}
We then recall from \cite{MW10a} that if Assumption \ref{ass_expmom} holds and if $(\hat{X},\hat{Y})$ is the solution to the primal and dual optimization problem, then $\hat{\Psi}\in\mathfrak{E}$. 
With regards to the derivation of the BSDE satisfied by $\hat{\Psi}$, we note that, using the formulae for $\hat{X}$ and $\hat{Y}$,
\begin{equation*}
\hat{\Psi}=\log\!\left(y\,\mc{E}(-\lambda\cdot M + \hat{N})\Big(x\,\mc{E}\left(\hat{\nu}\cdot M +\hat{\nu}\cdot \lo M,M \ro\lambda\right)\Big)^{1-p}\,\right).
\end{equation*}
After the change of variables \[\hat{Z}:=-\lambda+(1-p)\hat{\nu}\] a calculation shows that we have found a solution $(\hat{\Psi},\hat{Z},\hat{N})$ to the following quadratic semimartingale BSDE (written in the generic variables $(\Psi,Z,N)$),
\begin{multline}
d\Psi_t= Z_t^\tr\,dM_t + dN_t -\frac{1}{2}\,d\lo
N,N\ro_t
\label{BSDE}\\+\frac{q}{2}(Z_t+\lambda_t)^\tr\,d\lo
M,M\ro_t(Z_t+\lambda_t)-\frac{1}{2}\,Z_t^\tr\,d\lo
M,M\ro_tZ_t,\quad\Psi_T=0,
\end{multline}
where we refer to Definition \ref{DefnSol} for the notion of a \emph{solution} to the  BSDE \eqref{BSDE}. We summarize our findings in the following theorem, noting that it is uniqueness which requires the stronger Assumption \ref{ass_expmom}, existence is guaranteed via Assumption \ref{ass2}. 
\begin{thm}
\label{thmuniq} Let Assumption \ref{ass_expmom} hold and $(\hat{X},\hat{Y})$ be the solution pair to the primal/dual optimization problem i.e. for $x>0$
\begin{gather*}
\hat{X}=x\,\mc{E}(\hat{\nu}\cdot M +\hat{\nu}\cdot \lo M,M
\ro\lambda)\text{ and }\hat{Y}=u'(x)\,\mc{E}(-\lambda\cdot M +
\hat{N}).
\end{gather*}
Setting
\begin{equation*}
\hat{\Psi}:=\log\Bigl({\hat{Y}}\!\big/\,{U'(\hat{X})}\Bigr)\quad\text{ and }\quad
\hat{Z}:=-\lambda+(1-p)\hat{\nu},
\end{equation*} then
\begin{enumerate}
\item The triple $(\hat{\Psi},\hat{Z},\hat{N})$ is the unique solution $({\Psi},{Z},{N})$ to the BSDE \eqref{BSDE} where $\Psi\in\mathfrak{E}$ and $Z\cdot M$ and $N$ are two square-integrable martingales.
\item In terms of the BSDE we may write $\hat{Y}$ as
\begin{gather*}
\hat{Y}=\exp(\hat{\Psi})\,U'(\hat{X})=e^{\hat{\Psi}_0}x^{p-1}\,\mc{E}(-\lambda\cdot
M + \hat{N})\in\mc{Y}\bigl(c_p\,x^{p-1}\bigr)
\end{gather*}
with $c_p:=\exp\bigl(\hat{\Psi}_0\bigr)$, a.s.
\item The process $\mc{E}\Bigl(\bigl[(1-q)\hat{Z}-q\lambda\bigr]\cdot M+\hat{N}\Bigr)$ is a martingale on $[0,T]$.
\end{enumerate}
\end{thm}
\begin{proof}
The content of item (i) follows from Theorem \ref{thmMW10} in the Appendix which summarizes the main results on quadratic semimartingale BSDEs under an exponential moments condition. A calculation yields the alternative formula for $\hat{Y}$ in item (ii) and the relation \[e^{\hat{\Psi}_0}x^{p}\,\mc{E}\Bigl(\bigl[(1-q)\hat{Z}-q\lambda\bigr]\cdot M+\hat{N}\Bigr) \equiv\hat{X}\hat{Y}\] gives the remaining assertion in item (iii). 
\end{proof}

The statement of the above theorem is essentially known. In \cite{HIM05} and \cite{Mo09} the boundedness of the mean-variance tradeoff is used to ensure uniqueness, in \cite{MW10a} this argument is extended to the unbounded case with exponential moments. Building on \cite{MT03,MT08}, the article \cite{Nu209} shows that in a general setting the opportunity process $\exp(\hat{\Psi})$ satisfies a BSDE which reduces to \eqref{BSDE} under the additional assumption of continuity of the filtration. In particular, $\exp(\hat{\Psi})$ is identified there as the \emph{minimal} solution to the corresponding BSDE. 

Having identified candidate optimizers from the BSDE, a difficult task is then verification, i.e. showing that a solution to the BSDE indeed provides the primal and dual optimizers. A sufficient condition is that $\mc{E}\Bigl(\bigl[(1-q)Z-q\lambda\bigr]\cdot M+N\Bigr)$ is a martingale as can be derived from \cite{Nu209}, see also Proposition \ref{ExplSolW} below. However, given a solution $(\Psi,Z,N)$ to the BSDE \eqref{BSDE}, this condition need not be satisfied, hence a solution to the BSDE \eqref{BSDE} need \emph{not} yield the optimizers even when $Z$ and $N$ are square-integrable, as demonstrated in Subsection \ref{SecNonOptim}. In conclusion, if a solution triple $(\Psi,Z,N)$ exists, then under some conditions it provides the solution $(\hat{X},\hat{Y})$ to the primal and dual problem and we have uniqueness to the BSDE within a certain class. This is in the spirit of \cite{MT08} Theorem 1.3.2 and \cite{Nu209} Theorem 5.2. However, above and in their theorems, the requirements imposed are not on the model. In contrast, our goal is to study which conditions on the model, i.e. on $\lambda$ and $M$, ensure such a BSDE characterization and the regularity of its solution in terms of a \emph{bounded} dynamic value process.

\section{Existence, Uniqueness and Optimality for Quadratic BSDEs}
From Theorem \ref{thmuniq} we see that under Assumption \ref{ass_expmom} one can connect the duality and BSDE approaches to solving the utility maximization problem. To analyze this connection in further detail, we consider in the present section a setting where the BSDE \eqref{BSDE} is explicitly solvable. Proposition \ref{ExplSolW} gives a sufficient condition for a solution to the BSDE \eqref{BSDE} to exist and provides an expression for $\hat{\Psi}$ in terms of $\mc{E}(-\lambda\cdot M)$. 

We go on to study uniqueness and show in Theorem \ref{thm_continuum} that in general there are many solutions with square-integrable martingale part. This is a consequence of the fact that a multiplicative representation of random variables as stochastic exponentials need not be unique, which is the content of Lemma \ref{itoexplem}. Finally, a main aim in the present article is to study the boundedness of solutions to the BSDE \eqref{BSDE} under the exponential moment and BMO conditions. This involves constructing counterexamples and some of the key techniques and ideas used for this are introduced in the current section. Therefore we restrict ourselves to the \emph{Brownian setting}, which we assume to be one-dimensional for notational simplicity. Let $W$ be a one-dimensional Brownian motion under $\mathbb{P}$ and $(\mathcal{F}_t)_{t\in[0,T]}$ its augmented natural filtration. In particular, $N\equiv 0$ is the unique local martingale orthogonal to $M=W$. A generalization of the following results to the multidimensional Brownian framework is left to the reader.

\subsection{Necessary Conditions for the Existence of Solutions to Quadratic BSDEs}\label{SecExpMom} 
\begin{prop}\label{ExplSolW}
For $q\in[0,1)$ the BSDE \eqref{BSDE} always admits a solution. For $q<0$ the BSDE \eqref{BSDE} admits a solution if and only if
\begin{equation}\label{ExpSolWCond}
\E\Bigl[\bigl(Y^\lambda_T\bigr)^q \Bigr]=\E\bigl[ \mc{E}(- \lambda \cdot W)_T^q \bigr] < + \infty.
\end{equation}
If there exists a solution, there is a unique solution $(\hat{\Psi},\hat{Z})$ with $\mc{E}\bigl([(1-q)\hat{Z}-q\lambda ] \cdot W\bigr)$ a martingale. Its first component is given by
\begin{equation}\label{hatPsiExpl}
\hat{\Psi}_t = \frac{1}{1-q}\, \log\Bigl( \E\bigl[ \mc{E}(- \lambda \cdot W)_{t,T}^q \,\big|\, \mathcal{F}_t\bigr]\Bigr),\quad t\in[0,T],\quad \text{a.s.}
\end{equation}
In particular, solving \eqref{BSDE} and setting $(X,Y)$ as suggested by Theorem \ref{thmuniq} gives the pair of primal and dual optimizers.
\end{prop}
As a result, condition \eqref{ExpSolWCond} is sufficient for existence and uniqueness of the optimizers and we mention that it corresponds to condition (10) in \cite{KS99}. Hence, the utility maximization problem is well-defined even if NFLVR (\emph{no free lunch with vanishing risk}) does not hold. This is because FLVR strategies cannot be used beneficially by the CRRA-investor due to the requirement of having a positive wealth at any time.

\begin{proof}
Let us first show that the BSDE \eqref{BSDE} admits a solution if \eqref{ExpSolWCond} holds. Observe that from Jensen's inequality, for $q\in[0,1)$, \[\E\bigl[ \mc{E}(- \lambda \cdot W)_T^q \bigr] \leq \E\bigl[ \mc{E}(- \lambda \cdot W)_T \bigr]^q \leq 1,\] so that \eqref{ExpSolWCond} automatically holds in this case.
For $t\in[0,T]$ consider
\begin{equation}\label{auxM}
\ol{M}_{t} := \E\Biggl[\mc{E}(-q \lambda \cdot W)_T\exp\Biggl(\frac{q(q-1)}{2}\int_{0}^{T}  {\lambda}_{s}^{2}\,ds \Biggr)\Bigg|\, \mc{F}_{t} \Biggr].
\end{equation}
Since we have\[\E\Biggl[\mc{E}(-q \lambda \cdot W)_T\exp\Biggl(\frac{q(q-1)}{2} \int_{0}^{T} {\lambda}_{s}^{2}\,ds \Biggr)\Biggr]=\E\bigl[ \mc{E}(- \lambda \cdot W)_T^q \bigr] < + \infty,\] $\ol{M}$ is a positive martingale so that by It\^o's representation theorem there exists a predictable process $\ol{Z}$ with $\int_{0}^{T} \ol{Z}_{t}^{2}\, dt <+ \infty$ a.s. such that $ \frac{1}{\,\ol{M}}\cdot \ol{M}\equiv \ol{Z} \cdot {W}$. We set \[\hat{Z}:=\frac{\ol{Z}+q\lambda}{1-q}\] and $\hat{\Psi}$ as in \eqref{hatPsiExpl}. A calculation then shows that $(\hat{\Psi},\hat{Z})$ solves the BSDE \eqref{BSDE} with $\mc{E}\bigl([(1-q)\hat{Z}-q\lambda ] \cdot W\bigr)\equiv\mc{E}\bigl( \ol{Z} \cdot {W}\bigr) \equiv \tfrac{1}{\ol{M}_0}\,{\ol{M}}$ a martingale.

We now turn our attention to uniqueness. Let us assume that $({\Psi},{Z})$ is a solution to \eqref{BSDE} such that $\mc{E}\bigl( [(1-q){Z}-q\lambda] \cdot W\bigr)$ is a martingale. For $t\in[0,T]$ a calculation gives
\begin{align}\label{PsiZWlamExpl}
\exp\bigl(-(1-q)\Psi_t\bigr)&=\exp\bigl((1-q)({\Psi}_T-{\Psi}_t)\bigr)\\\notag&=\mc{E}\bigl( [(1-q){Z}-q\lambda] \cdot W\bigr)_{t,T}\,\mc{E}\bigl(- \lambda \cdot W\bigr)_{t,T}^{-q}\quad \text{a.s.}
\end{align}
so that we obtain
\[{\Psi}_t = \frac{1}{1-q}\, \log\Bigl( \E\bigl[ \mc{E}(- \lambda \cdot W)_{t,T}^q \,\big|\, \mathcal{F}_t\bigr]\Bigr)\quad \text{a.s.}\]
We derive that $\Psi$ and $\hat{\Psi}$ are indistinguishable due to continuity. From \eqref{PsiZWlamExpl} we then obtain that $\mc{E}\bigl( [(1-q){Z}-q\lambda] \cdot W\bigr)$ is uniquely determined, from which it follows that $\hat{Z} \cdot{W}\equiv Z\cdot W $.

Finally, we show that the condition \eqref{ExpSolWCond} is also necessary. Assume that a solution $(\Psi,Z)$ to \eqref{BSDE} exists but $\E\bigl[ \mc{E}(- \lambda \cdot W)_T^q \bigr] = + \infty$. Then, together with the supermartingale property of $\mc{E}\bigl( [(1-q){Z}-q\lambda] \cdot W\bigr)$, the equality \eqref{PsiZWlamExpl} shows that a.s.
\begin{align*}
\exp\bigl((1-q)\Psi_0\bigr)&\geq\E\Bigl[\mc{E}\bigl( [(1-q){Z}-q\lambda] \cdot W\bigr)_{T}\Bigr]\,\exp\bigl((1-q)\Psi_0\bigr)\\&=\E\Bigl[\mc{E}\bigl(- \lambda \cdot W\bigr)_{T}^{q}\Bigr]=+\infty,
\end{align*}
from which $\Psi_0=+\infty$ a.s. in contradiction to the existence of $\Psi$.
\end{proof}
 
We now provide an explicit market price of risk for which condition \eqref{ExpSolWCond} fails to hold, hence for which the BSDE \eqref{BSDE} has no solution. While similar examples have been provided in Delbaen and Tang \cite{DT10} Theorem 2.8 as well as Frei and dos Reis \cite{FdR10} Lemma A.1, we give the full construction as it will be used throughout.

\begin{prop}\label{propnosol}
For every $q < 0$ there exists $\lambda$ such that $\lambda \cdot W $ is a bounded martingale and $\E\bigl[\bigl(Y^\lambda_T\bigr)^q \bigr]=\E\bigl[ \mc{E}(- \lambda \cdot W)_T^q \bigr] = + \infty.$
\end{prop}
\begin{proof}
For $t\in[0,T]$ define
\begin{equation*} 
\lambda_t:=\frac{\pi}{2\sqrt{-q(T-t)}}\,\mathbf{1}_{]\!]T/2,\tau]\!]}(t,\cdot),
\end{equation*}
where $\tau$ is the stopping time \[ \tau := \inf\Bigg\{ t >\frac{T}{2}\,\Bigg|\, \Bigg| \int_{T/2}^{t} \frac{1}{\sqrt{T-s}}\,dW_{s}\Bigg|\geq1 \Bigg\}.\] Here, we define $\lambda$ from time $T/2$ onwards to be consistent with the construction in Subsection \ref{CountBound3}. For the present proof, we could equally well replace $T/2$ by $0$ in the definitions of $\lambda$ and $\tau$. Observe that we have that $\mb{P}(T/2<\tau<T)=1$ due to continuity and the relation \[\left\lo\int_{T/2}^\cdot\frac{1}{\sqrt{T-t}}\,dW_t,\int_{T/2}^\cdot\frac{1}{\sqrt{T-t}}\,dW_t\right\ro_T=\int_{T/2}^T\frac{1}{T-t}\,dt=+\infty.\]
By construction, $\lambda\cdot {W}$ is bounded by $\frac{\pi}{2\sqrt{-q}}$. We obtain, using \cite{Ka94} Lemma 1.3 similarly to the proof of \cite{FdR10} Lemma A.1,
\begin{align*}
\E\Bigl[ \mc{E}(- \lambda \cdot W)_T^q \Bigr]&=\E\Biggl[\exp\Biggl(-q (\lambda \cdot W)_T-\frac{q}{2} \int_{0}^{T} {\lambda}_{t}^{2}\,dt \Biggr)\Biggr]\\&\geq e^{-\frac{\pi\sqrt{-q}}{2}}\,\E\Biggl[\exp\Biggl(\frac{\pi^2}{8} \int_{T/2}^{\tau} \frac{1}{T-t}\,dt \Biggr)\Biggr]=+\infty,
\end{align*}
from which the statement follows immediately.
\end{proof}
Let us make two points concerning the above example, firstly that when $q\in[0,1)$ such a degeneracy cannot occur, as shown by Proposition \ref{ExplSolW}. In fact for a BMO martingale $\lambda\cdot W$ there actually always exists a (then unique) \emph{bounded} solution, as Corollary \ref{cornew} below shows. We recall from \cite{Ka94} that a continuous martingale $\ol{M}$ on the compact interval $[0,T]$ with $\ol{M}_0=0$ is a \emph{BMO martingale} if \[\big\|\ol{M}\big\|_{\mathrm{BMO}_2}:=\sup_{\tau}\bigg\|\E\Big[\bigl(\ol{M}_T-\ol{M}_\tau\bigr)^2\Big|\,\mc{F}_\tau\Big]^{1/2}\bigg\|_{L^\infty}<+\infty,\] where the supremum is over all stopping times $\tau$ valued in $[0,T]$.

Secondly we point out that the martingale in the above proposition is bounded. Indeed, it is a leitmotiv of the present article that requiring (in addition) the martingale $\lambda\cdot M$ to be bounded does not improve the situation with respect to finiteness of a BSDE solution. This is because the key estimates are all on the quadratic variation process $\lo\lambda \cdot M,\lambda \cdot M \ro$ which in general does not inherit such properties.

\begin{rmk} As described in the introduction, the assumption that $\lambda\cdot M$ is a BMO martingale is natural from a no-arbitrage perspective. We now give an additional financial interpretation of the BMO condition. Suppose for simplicity that $S$ is a geometric Brownian motion of the form
\begin{equation*}
dS_t= S_t \,\bigl( \sigma_{t}\, dW_t+ \mu_{t} \,dt),
\end{equation*}
 so that $\lambda =  \mu / \sigma^{2} $ and  $\lo\lambda \cdot M,\lambda \cdot M \ro = \int_0^\cdot \frac{\mu^{2}_t}{ \sigma^{2}_t} \, dt$. The Sharpe ratio, defined as  $\mu  / \sigma$, measures the return per unit of risk. A BMO condition on $\lambda \cdot M$  requires that the conditional expected integral of the squared Sharpe ratio $\E\bigl[ \int_{\cdot}^{T} \frac{\mu^{2}_{t}}{ \sigma^{2}_{t}} \, dt \big|\, \mathcal{F}_{\cdot} \bigr]$ be bounded, which implies a restriction on the asset not offering huge returns with tiny risk.

Developing this idea further, consider an investment strategy $\pi$ which in this remark represents the amount, not, as elsewhere, the proportion, of wealth invested in $S$. We assume that \mbox{$\E\bigl[\lo \lambda \cdot M,\lambda \cdot M\ro_T \bigr]<+\infty$} and that $\pi$ is predictable and satisfies \mbox{$\E\bigl[\lo \pi \cdot M,\pi \cdot M\ro_T \bigr]<+\infty$}. From the stock dynamics
\begin{equation*}
dS_t=\text{Diag}(S_t)\,\bigl(dM_t+ \,d\lo M,M \ro_t\lambda_t\bigr),
\end{equation*}
it follows that the expected gain (or loss) related to $\pi$ is given by
\begin{equation*}
\E\biggl[\sum_{i = 1}^{d} \int_{0}^{T} \frac{ \pi_{t}^{i}}{S_{t}^{i}} \, dS_t^{i}\biggr]=\E\biggl[\int_{0}^{T} \pi_{t}  \, dM_t \biggr]  +\E\biggl[\int_{0}^{T} \pi_{t}^{\tr} \,d\lo M,M \ro_t\lambda_t\biggr]= \E\biggl[\int_{0}^{T} \pi_{t}^{\tr} \,d\lo M,M \ro_t\lambda_t\biggr].
\end{equation*}
Using \cite{Pr05} Theorem IV.54, we deduce that $\lambda\cdot M$ is  a BMO martingale if and only if there exists a constant $c$ such that
\begin{equation} \label{H1bmo}
\Bigg| \,\E\Biggl[ \sum_{i = 1}^{d} \int_{0}^{T} \frac{ \pi_{t}^{i}}{S_{t}^{i}} \, dS_t^{i}\Biggr] \Bigg| \leq c\, \E\Bigl[ \lo \pi \cdot M,\pi \cdot M\ro_T^{1/2}\Bigr] =: c\, \|  \pi \cdot M \|_{\mathcal{H}^{1}} 
\end{equation}
for all $\pi$ with \mbox{$\E\bigl[\lo \pi \cdot M,\pi \cdot M\ro_T \bigr]<+\infty$}. Using \[\bigg\lo \sum_{i = 1}^{d} \int^\cdot_0 \frac{ \pi^{i}_t}{S^{i}_t} \, dS^{i}_t,\sum_{i = 1}^{d} \int_0^\cdot \frac{ \pi^{i}_t}{S^{i}_t} \, dS^{i}_t\bigg\ro_T = \lo \pi \cdot M,\pi \cdot M\ro_T,\] we can view  $\|  \pi \cdot M \|_{\mathcal{H}^{1}} $ as a measure of risk in our portfolio. We see from \eqref{H1bmo} that the assumption of $\lambda\cdot M$ being a BMO martingale means that portfolios with bounded risk (in this $\mathcal{H}^{1}$-sense) always have bounded expected gains. Conversely, if the expected gains are bounded in terms of such risk uniformly over all investment strategies, then $\lambda\cdot M$ needs to be a BMO martingale.% Conversely, if all portfolios with such a bounded risk have uniformly bounded expected gains, then $\lambda\cdot M$ needs to be a BMO martingale.
\end{rmk}

\subsection{Nonoptimality of BSDE Solutions}\label{SecNonOptim}
If a solution to the BSDE \eqref{BSDE} does exist, it does \emph{not} automatically lead to an optimal pair for the utility maximization problem. This is because it may fail to be in the right space (e.g. with respect to which uniqueness for BSDE solutions holds). We now provide a theoretical result to illustrate the problem. More precisely, in contrast to the classical It\^{o} representation theorem with square-integrable integrands, an analogous representation of random variables in terms of stochastic exponentials is not unique.
\begin{lem}\label{itoexplem}
Let $\xi$ be a random variable bounded away from zero and infinity, i.e. there are constants $L,\ell>0$ such that $\ell\leq\xi\leq L$ a.s. Then, for every real number $c \geq \E [ \xi ]$, there exists a predictable process $\alpha^{c}$ such that
\begin{equation}\label{itoexp}
\xi = c\, \mathcal{E}(\alpha^{c} \cdot W)_{T}, \quad \E\!\left[\int_{0}^{T} |\alpha_{t}^{c}|^{2} \, dt \right] <+\infty.
\end{equation}
However, there is only one pair $(c,\alpha )$ which satisfies $\xi = c\, \mathcal{E}(\alpha \cdot W)_{T}$ with $\alpha \cdot W$ a BMO martingale or, equivalently, with $c = \E[\xi]$.
\end{lem}
\begin{rmk}
Comparing the multiplicative representation \eqref{itoexp} with the classical one, see \cite{KS91} Theorem 4.15, namely
\[\xi = k +  (\beta  \cdot W)_{T}, \quad \E\!\left[\int_{0}^{T} |\beta_{t}| ^{2} \, dt \right] <+\infty,\]
we see that existence holds in both cases, whereas there is no uniqueness of $(c,\alpha^{c})$ in \eqref{itoexp} despite the fact that $\E\!\left[\int_{0}^{T} |\alpha_{t}^{c}|^{2} \, dt \right] <+\infty$, in contrast to the uniqueness of $(k,\beta )$. 

While in the standard It\^{o} representation for $L^2$-random variables the square-integra\-bility and martingale property are equivalent, our result shows that in the multiplicative form $\E\!\left[\int_{0}^{T} |\alpha_{t}^{c}|^{2} \, dt \right] <+\infty$  does not guarantee uniqueness. The intuition for the difference between $\beta$ in the additive and $\alpha^{c}$ in the multiplicative form is the following. Since $\beta$ is a square-integrable process, $\beta\cdot W$ is a martingale, hence it \emph{must} be the case that $k=\E[\xi]$. In contrast, the square-integrability of $\alpha^c$ is not sufficient for $\mc{E}(\alpha^c\cdot W)$ to be a martingale. Indeed, it can be that $\E[\mc{E}(\alpha^c\cdot W)_T]<1$ so that increasing $c\geq\E[\xi]$ may be offset by an appropriate choice of $\alpha^c$ such that \eqref{itoexp} still holds. A consequence of this is that uniqueness of $\xi = c\, \mathcal{E}(\alpha \cdot W)_{T}$ holds if $\alpha  \cdot W$ is a BMO martingale or equivalently (see \cite{Ka94} Theorem 3.4, using the boundedness of $\xi$) if $\mc{E}(\alpha  \cdot W)$ is a martingale.

One could argue that a more natural condition in \eqref{itoexp} is to assume that $\mathcal{E}(\alpha^{c} \cdot W)$ be a true martingale, however our aim is a characterization in terms of $\alpha^{c} \cdot W$ itself and thus we do not pursue this. Note that it is not possible to find $c <  \E [ \xi ]$ such that \eqref{itoexp} holds, because $\mathcal{E}(\alpha^{c} \cdot W)$ is always a positive local martingale, hence a supermartingale.
\end{rmk}
\begin{proof}  We first define $\ol{M}_{t}:= \E[ \xi | \mathcal{F}_{t}]$, $t \in [0,T]$, and apply  It\^o's representation theorem to the stochastic logarithm of $\ol{M}$, which is a BMO martingale by \cite{Ka94} Theorem 3.4 since $\ol{M}$ is bounded away from zero and infinity. This application yields a predictable process $\ol{\alpha}$ such that $\ol{\alpha} \cdot W$ is a BMO martingale and $\xi =\E [ \xi ]\, \mathcal{E}(\ol{\alpha} \cdot W)_{T}$. The uniqueness part of the statement is then immediate; if $ \alpha  \cdot W$ is a BMO martingale, we have $c=\E[\xi]$ and $\alpha  \cdot W \equiv \ol{\alpha} \cdot W$ since $\mc{E}(\alpha  \cdot W)$ is a martingale. Conversely, if $c=\E[\xi]$ the process $\mc{E}(\alpha  \cdot W)$ is a supermartingale with constant expectation, hence a martingale. Indeed, we then have 
\[\mc{E}(\alpha  \cdot W)\equiv \E[ \mc{E}({\alpha}  \cdot W)_T | \,\mc{F}_{.}]\equiv\frac{1}{\E[\xi]} \,\E[ \xi | \,\mc{F}_{.}] \equiv \E[ \mc{E}(\ol{\alpha}  \cdot W)_T |\, \mc{F}_{.}] \equiv\mc{E}(\ol{\alpha}  \cdot W)\] and thus $\alpha \cdot W\equiv\ol{\alpha}  \cdot W$, which is the BMO martingale from above.

To construct $\alpha^c$ we fix $c \geq \E [ \xi ]$ and define the stopping time
\[\tau_c :=\inf\bigg\{t\geq0~ \bigg|~\int_{0}^{t}\frac{1}{T-s}\,dW_{s} \leq \frac{t}{2T(T-t)} + \log \frac{\ol{M}_{t}}{c}  \bigg\}.\]
We argue that $ \tau_c < T$ a.s. To this end consider 
\[\ol{\tau}_{c} := \inf\bigg\{t\geq0~ \bigg|~\int_{0}^{t}\frac{1}{T-s}\,dW_{s} \leq \frac{t}{2T(T-t)}+ \log \frac{ \ell}{c}  \bigg\}\] and observe that $\tau_c \leq \ol{\tau}_{c}$. If we define the time change $\rho: [0,T] \to [0,+\infty]$ by $\rho(t) := \frac{t}{T(T-t)}$, then it follows from  \cite{RY99} II.3.14. that
\begin{equation}\label{tauexpmom}
\E\biggl[\exp\biggl(\frac{1}{8}\,\rho(\ol{\tau}_c)\biggr)\biggr]=\exp\!\left( -\frac{1}{2}\log \frac{\ell}{c} \right)  = \frac{c^{1/2}  }{ \ell^{1/2} } < + \infty.
\end{equation}
We deduce that $\E[\rho(\ol{\tau}_c)] < +\infty$, $\rho(\ol{\tau}_c) < +\infty$ a.s.~and $\ol{\tau}_c < T$ a.s. from which it follows that indeed $ \tau_c < T$ a.s. 

We now define \[\alpha^{c}_{t} := \frac{1}{T-t}\, \mathbf{1}_{[\![0,\tau_c]\!]}(t,\cdot)+ \ol{\alpha} \,\mathbf{1}_{]\!]\tau_c,T]\!]}(t,\cdot),\]
which satisfies
\[c \, \mathcal{E}(\alpha^{c} \cdot W)_{T} = c \, \frac{\ol{M}_{T}}{\ol{M}_{\tau_c}}\, \mathcal{E}(\alpha^{c} \cdot W)_{\tau_c}  = c\, \frac{\ol{M}_{T}}{\ol{M}_{\tau_c}}\frac{\ol{M}_{\tau_c}}{c}= \ol{M}_{T} = \xi,\]
where the second equality is due to the specific definition of the stopping time $\tau_c$. Moreover, we have
\[\E\!\left[ \int_{0}^{T}|\alpha^{c}_{t} |^{2}\,dt\right]\leq\E\!\left[\int_{0}^{\tau_c} \frac{1}{(T-t)^{2}}\,dt\right] + \E\!\left[ \int_{0}^{T}|\ol{\alpha}_{t} |^{2}\,d t\right] = \E[\rho (\tau_c)]+ \E\!\left[ \int_{0}^{T}|\ol{\alpha}_{t} |^{2}\,d t\right]\!,\]
which is finite because $ \E[\rho(\tau_c)] \leq \E[\rho(\ol{\tau}_c)] < +\infty$ and $\ol{\alpha} \cdot W$ is a BMO martingale.
\end{proof}
The standard method of finding solutions to quadratic BSDEs involves an exponential change of variables. A consequence of the preceding lemma is that the above type of nonuniqueness transfers to the corresponding BSDE solutions, in particular to those of the utility maximization problem. Observe that for each $c$ the process $\alpha^c\cdot W$ is square-integrable in contrast to classical locally integrable counterexamples. Indeed it is well known that without square-integrability even the standard It\^{o} decomposition is not unique. In fact, for every $k \in \mathbb{R}$ there exists $\beta^{k}$ such that
\[\xi = k +  (\beta^{k}  \cdot W)_{T}, \quad \int_{0}^{T} \big|\beta_{t}^{k}\big| ^{2} \, dt   <+\infty\quad\textrm{ a.s.}\] see \cite{ESY83} Proposition 1. We are hence able to construct distinct solutions to the BSDE \eqref{BSDE}. This amounts to some of those solutions being nonoptimal by Proposition \ref{ExplSolW}. 
Alternatively, uniqueness of the multiplicative decomposition $\xi = c\, \mathcal{E}(\alpha \cdot W)_{T}$ holds under an additional BMO assumption which then implies the uniqueness of $\hat{\Psi}$. We summarize these comments in the following theorem.

\begin{thm}\label{thm_continuum}
For all $p\in(-\infty,1)$ and $\lambda$ with $\int_0^T\lambda^2_t\,dt$ bounded, there exists a continuum of distinct solutions $({\Psi}^b,{Z}^b,{N}^b\equiv0)$ to the BSDE \eqref{BSDE}, parameterized by $b\geq0$, satisfying the following properties:
\begin{enumerate}
\item The martingale part ${Z}^b\cdot{W}$ is square-integrable for all $b\geq0$.
\item The process $\mc{E}\Bigl(\bigl[(1-q){Z}^b-q\lambda\bigr]\cdot W\Bigr)$ is a martingale if and only if $b=0$.
\item Defining $\nu^b$ as suggested by the formula in Theorem \ref{thmuniq}, the admissible process $\nu^b$ is the optimal strategy if and only if $b=0$. 
\end{enumerate} 
\end{thm}
It is known from \cite{AIP08} Section 2.2 that quadratic BSDEs need not
have unique square-integrable solutions. These authors present a specific
example of a quadratic BSDE which allows
for distinct solutions with square-integrable martingale part. In contrast,
Theorem \ref{thm_continuum} shows that every BSDE related to power
utility maximization with bounded mean-variance tradeoff has no unique square-integrable solution, independently of the value of $p$. This underlines the importance of being able to find a solution to the BSDE \eqref{BSDE} with $Z\cdot W$ a BMO martingale in \cite{HIM05} and \cite{Mo09}.
\begin{proof}
We set $\xi:=\exp\Bigl(\frac{q(q-1)}{2}\int_0^T\lambda^2_t\,dt\Bigr)$ and define the measure change \[\frac{d\ti{\mathbb{P}}}{d\mathbb{P}}:=\mathcal{E}(-q\lambda\cdot{W})_T,\] so that $\ti{\mathbb{P}}$ is an equivalent probability measure under which $\widetilde{W}$ is a Brownian motion on $[0,T]$ where \[\widetilde{W}_t:={W}_t+q\int_0^t\lambda_s\,ds.\] Observe that this measure change is implicitly already present in the proof of Proposition \ref{ExplSolW}, see \eqref{auxM}. We now apply Lemma \ref{itoexplem} to the triple $\bigl(\ti{W},\ti{\mathbb{P}},(\mc{F}_t)_{t\in[0,T]}\bigr)$ noting that in its proof we may use It\^{o}'s representation theorem in the form of \cite{KS98} Theorem 1.6.7, i.e. we can write any $\ti{\mathbb{P}}$-martingale as a stochastic integral with respect to $\ti{W}$, although $\ti{W}$ may not generate the whole filtration $(\mc{F}_{t})_{t\in[0,T]}$. For every real number $c\geq\E_{\ti{\mathbb{P}}}[\xi]$ we then derive the existence of a predictable process $\alpha^c$ such that
\begin{equation*}
\xi = c\, \mathcal{E}(\alpha^{c} \cdot \ti{W})_{T}, \quad \E_{\ti{\mathbb{P}}}\!\left[\int_{0}^{T} |\alpha_{t}^{c}|^{2} \, dt \right] <+\infty.
\end{equation*}
For $t\in[0,T]$ we then set \[\ti{\Psi}_t^c:=\log(c)+\int_0^t\alpha^c_s\,d\ti{W}_s-\frac12\int_0^t|\alpha_s^c|^2\,ds-\frac{q(q-1)}{2}\int_0^t\lambda_s^2\,ds,\] so that $\ti{\Psi}^c$ solves the BSDE 
\begin{equation*} 
d\ti{\Psi}_{t}^c=\alpha^c_t\,d\ti{W}_t-\frac12\,|\alpha_t^c|^2\,dt-\frac{q(q-1)}{2}\,\lambda_t^2\,dt,\quad\ti{\Psi}^c_T=0.
\end{equation*}
Using the transformations $b:=c-\E_{\ti{\mathbb{P}}}[\xi]\geq0$, $\ti{\Psi}^c=:(1-q){\Psi^b}$ and $\alpha^c=:(1-q){Z^b}$ we arrive at the BSDE \eqref{BSDE},
\[d{\Psi}^b_t={Z}^b_t\,d{W}_t+\frac{q}{2}({Z}^b_t+\lambda_t)^2\,dt-\frac{1}{2}\,\bigl({Z}_t^b\bigr)^2\,dt,\quad{\Psi}^b_T=0,\]
which admits a continuum of distinct solutions, parameterized by $b\geq0$. We show that each martingale part is additionally square-integrable under ${\mathbb{P}}$. This follows from the inequality \[\E\!\left[\int_0^T|\alpha^c_t|^2\,dt\right]\leq\E_{\ti{\mathbb{P}}}\!\left[\mathcal{E}(q\lambda\cdot\ti{W})^{2}_T\right]^{1/2}\E_{\ti{\mathbb{P}}}\!\left[\left(\int_0^T|\alpha^c_t|^2\,dt\right)^2\right]^{1/2}\!\!.\]
Note that the second term on the right hand side is finite since from \eqref{tauexpmom} in the proof of Lemma \ref{itoexplem} we have that $\E_{\ti{\mathbb{P}}}\bigl[\rho(\ol{\tau}_c)^2\bigr]<+\infty$. Moreover, using $\ol{\alpha}$ from this proof, $\ol{\alpha}\cdot \ti{W}$ is a BMO martingale (under $\ti{\mathbb{P}}$), hence $\int_0^T|\ol{\alpha}_t|^2\,dt$ has an exponential $\ti{\mathbb{P}}$-moment of some order by \cite{Ka94} Theorem 2.2, see also Lemma \ref{KaJohnNirenL} and the comments thereafter. To derive that the first term is finite we use that
\begin{equation*}
\E_{\ti{\mathbb{P}}}\!\left[\mathcal{E}(q\lambda\cdot{\ti{W}})^{2}_T\right]\leq\E_{\ti{\mathbb{P}}}\!\left[\exp\!\left(6q^2\int_0^T\lambda^2_t\,dt\right)\right]^{1/2}<+\infty.
\end{equation*}
Clearly, the Assumption \ref{ass_expmom} is satisfied, hence our previous analysis applies. However, there is a continuum of distinct solutions $({\Psi}^b,{Z}^b)$ to the BSDE \eqref{BSDE} since for every $b=c-\E_{\ti{\mathbb{P}}}[\xi]\geq0$ we have that ${\Psi}^b_0=\frac{\log(c)}{1-q}$. From \cite{Ka94} Theorem 3.6 we have that $\alpha^c\cdot\ti{W}$ is a BMO martingale under $\ti{\mathbb{P}}$ if and only if $\alpha^c\cdot{W}$ is a BMO martingale under ${\mathbb{P}}$. This last condition holds if and only if $Z^b\cdot W$ is a BMO martingale under ${\mathbb{P}}$. We conclude that $\mc{E}\bigl(\bigl[(1-q){Z}^b-q\lambda\bigr]\cdot W\bigr)$ is a martingale if $b=0$. It cannot be a martingale for $b>0$ since otherwise $(\Psi^b,Z^b)$ would coincide with $(\hat{\Psi},\hat{Z})\equiv(\Psi^0,Z^0)$. The last assertion is then immediate.
\end{proof}

%%%%%%%%%%%%%%%%%%%%%%%%%%%%%%%%%%%%%%%%%%%%%%%%%%%%%%%%%%%%%
%Boundedness of BSDE solutions
%%%%%%%%%%%%%%%%%%%%%%%%%%%%%%%%%%%%%%%%%%%%%%%%%%%%%%%%%%%%%

\section{Boundedness of BSDE Solutions and the BMO Property}\label{sec5}
Thus far we have worked under an exponential moments assumption on the mean-variance tradeoff which provides us with the existence of the primal and dual optimizers as well as a link between these optimizers and a special quadratic BSDE. We now connect the above study to the boundedness of solutions to quadratic BSDEs which we show to be intimately related to the BMO property of the martingale part and the mean-variance tradeoff.

For $q\in[0,1)$ and under the assumption that $\lambda\cdot M$ is a BMO martingale, we show that the BSDE \eqref{BSDE} has a bounded solution. This follows as a consequence of existence results and a priori estimates for a general class of BSDEs as described below. We consider the  BSDE 
\begin{equation}
d\Psi_t= Z_t^\tr\,dM_t + dN_t -F(t,\Psi_{t},Z_t)\,dA_t-\frac{1}{2}\,d\lo N,N\ro_t ,\quad \Psi_T=\xi,\label{BSDEF2}
\end{equation}
where, as in the appendix, $A$ is a nondecreasing bounded process such that $\lo M,M\ro=B^{\tr}B\cdot A$ for a predictable process $B$ valued in the space of $d\times d$ matrices. We assume that $\xi$ is a bounded random variable and that the driver $F$ is continuous in $(\psi,z)$ and satisfies
\begin{gather}\label{condF} 
-\frac{\delta}{2}  \| B_t z \|^2 - \| B_t \kappa_t \|^2 \leq F(t,\psi,z) \leq   \phi(|\psi|) + \frac{\gamma}{2}  \| B_t z \|^2 + \| B_t \eta_{t} \|^2 , \\  \psi \big(F(t,\psi,z) - F(t,0,z) \big) \leq \beta |\psi|^{2};\nonumber
\end{gather}
$\beta, \delta,\gamma \geq 0$ are constants, $\phi$ is a deterministic continuous nondecreasing function with $\phi(0) =0$, and  $\kappa$, $\eta$ are processes such that $\kappa \cdot M$, $\eta \cdot M$ are BMO martingales. If $\beta \neq 0$, we additionally assume that there exists a constant $c_{A}$ such that $A_{t} \leq c_{A} \cdot t$ for all $t \in [0,T]$, and we set $\beta^{\star} = c_{A}\cdot \beta$.  For the notion of solution to \eqref{BSDEF2}, see Appendix \ref{appendBSDE}.

\begin{thm} \mbox{}\label{solbdd}
\begin{enumerate}
\item If $\delta = 0$ and $\eta$ satisfies \hbox{$\| \eta \cdot M \|_{\mathrm{BMO}_{2}}^{2}  < \frac{e^{-T \beta^{\star} }}{\max\{1,\gamma \}} $}, then the BSDE \eqref{BSDEF2} has a solution with bounded first component.
\item Assume $\phi \equiv 0$ and that $(\Psi, Z,N)$ is a solution to \eqref{BSDEF2}. Then 
$\Psi$ is bounded if and only if both $Z \cdot M$ and $N$ are BMO martingales.
\item Suppose $\phi \equiv 0$, $\delta = 0$ and that $\eta$ satisfies \hbox{$\| \eta \cdot M \|_{\mathrm{BMO}_{2}}^{2}  < \frac{e^{-T \beta^{\star} }}{\max\{1,\gamma \}} $}. Then there exists a solution $(\Psi, Z,N)$ with bounded $\Psi$ and BMO martingales $Z \cdot M$ and $N$.
\end{enumerate}
\end{thm}
\begin{proof} (i) Using the assumption $\delta = 0$, we can argue similarly to the proof of \cite{DHR09} Theorem 2.1. That proof is given in a Brownian setting but translates correspondingly  to our semimartingale model similarly to \cite{MW10}. With $\tilde{\gamma} :=  \max\{1,\gamma\}$ and based on \eqref{condF}, this yields the existence of a solution $(\Psi,Z,N)$  satisfying
\begin{align*}
\Psi_{t} & \leq  \frac{1}{\tilde{\gamma} } \log \E \bigg[\exp\bigg( \tilde{\gamma} e^{ (T-t)\beta^{\star} }\xi^{+} + \tilde{\gamma} \int_{t}^{T}   e ^{  (s-t)\beta^{\star} } \| B_s \eta_{s} \|^2 \, d A_{s}  \bigg) \bigg| \,\mathcal{F}_{t}\bigg], \\
\Psi_{t} & \geq  \E \bigg[  \xi  - \int_{t}^{T} \| B_s \kappa_s \|^2  \,dA_s  \bigg| \,\mathcal{F}_{t}\bigg].
\end{align*}
Since $\xi \in L^{\infty}$ and $\kappa  \cdot M$ is a BMO martingale, the latter inequality shows that $\Psi$ is bounded from below. From the former inequality, we obtain
\begin{align*}
\Psi_{t}  & \leq   e^{  T \beta^{\star} } \| \xi \|_{L^{\infty}}  +  \frac{1}{\tilde{\gamma} } \log \E \bigg[\exp\bigg(  \tilde{\gamma} e ^{ T\beta^{\star} }  \int_{t}^{T}    \| B_s \eta_{s} \|^2\, d A_{s}  \bigg) \bigg| \,\mathcal{F}_{t}\bigg] \\ & \leq e^{  T \beta^{\star}  } \| \xi \|_{L^{\infty}} -  \frac{1}{\tilde{\gamma}}  \log \bigl(1-   \tilde{\gamma} e ^{  T \beta^{\star} }  \| \eta\cdot M \|^2_{\mathrm{BMO}_{2}}\bigr)  
\end{align*}
by the John-Nirenberg inequality; see Lemma  \ref{KaJohnNirenL}.
This shows that $\Psi$ is bounded from above due to the assumptions $\xi \in L^{\infty}$ and $\| \eta \cdot M \|_{\mathrm{BMO}_{2}}^{2} < \frac{e^{-T \beta^{\star} }}{\max\{1,\gamma \}} $.  Hence, $\Psi$ is bounded, which concludes the proof of (i).

For (ii), let first $(\Psi, Z,N)$ be a solution to \eqref{BSDEF2} with bounded $\Psi$.  The BMO properties of $Z \cdot M$ and $N$ follow from  \cite{MS05} Proposition 7, using that \eqref{condF} with $\phi \equiv 0$ implies
\[
 |F(t,\psi,z)| \leq     \frac{ \gamma+\delta }{2}  \| B_t z \|^2 + \| B_t \eta_{t} \|^2+ \| B_t \kappa_{t} \|^2.
\]
Conversely, if  $Z\cdot M$ and $N$ are BMO martingales, then $\Psi$ is bounded. This follows by taking the conditional $t$-expectation in the integrated version of \eqref{BSDEF2} and estimating the remaining finite variation parts with the help of the $\mathrm{BMO}_2$ norms of $\kappa \cdot M$, $\eta \cdot M$, $Z\cdot M$ and $N$, uniformly in $t$.

Finally, the statement (iii) is an immediate consequence of the items (i) and (ii).
\end{proof}

 Let us now apply this result to the specific BSDE \eqref{BSDE} related to power utility maximization.
 
 \begin{cor} \label{cornew} Assume that $q\in[0,1)$ and that $\lambda\cdot M$ is a BMO martingale. Then the BSDE \eqref{BSDE} has a solution with bounded first component. In particular, setting $(X,Y)$ as suggested by Theorem \ref{thmuniq} gives the pair of primal and dual optimizers.
\end{cor} 
\begin{proof} By Appendix \ref{appendBSDE}, the BSDE \eqref{BSDE} is of the form \eqref{BSDEF2} with the driver given~by
\[
 F(t,z)=-\frac{q}{2}(z+\lambda_t)^\tr B^{\tr}_tB_t(z+\lambda_t)+\frac{1}{2}z^\tr B^{\tr}_tB_tz.
\]
Using $q\in[0,1)$, we can show by an elementary calculation that  
\[
\frac{-q}{2(1-q)} \|B_t\lambda_{t}\|^{2}   \leq F(t,z) \leq \frac{1}{2} \|B_tz\|^{2}
\]
so that \eqref{condF} is satisfied. Thus we can apply Theorem \ref{solbdd} (iii) to obtain that there exists a solution triple $(\Psi,Z,N)$ such that the process $Z\cdot M+N$ is a BMO martingale. We derive that $\bigl[(1-q)Z-q\lambda\bigr]\cdot M+N$ is a BMO martingale which, by \cite{Ka94} Theorem 2.3, shows that $\mc{E}\bigl(\bigl[(1-q)Z-q\lambda\bigr]\cdot M+N\bigr)$ is a true martingale. We then deduce that solving the BSDE with a bounded $\Psi$ gives rise to an optimal pair for the primal and dual problem.
\end{proof}

\begin{rmk} Instead of applying Theorem \ref{solbdd}, which holds for a more general class of BSDEs, Corollary \ref{cornew} can also be shown as follows using specific results related to power utility maximization. Since $\lambda\cdot M$ is a BMO martingale, $Y^\lambda:=\mc{E}(-\lambda\cdot M)$ defines an equivalent local martingale measure for $S$ so that Assumption \ref{ass2} is satisfied, where we use an easy calculation to extend its item (ii) to $p=0$. By \cite{Ka94} Corollary 3.4, if $q\in(0,1)$, the process $Y^\lambda\in\mc{Y}(1)$ satisfies the reverse H\"{o}lder inequality. This means that there is constant $c_{\textit{\tiny{rH}},p}>0$ (which depends on $p = \frac{q}{q-1}$) such that for all stopping times $\tau$ valued in $[0,T]$,
$$
\E\!\left[\left(Y^\lambda_T\big/Y^\lambda_\tau\right)^q\bigg|\,\mc{F}_\tau\right]\geq
c_{\textit{\tiny{rH}},p}
$$
The assertion of Corollary \ref{cornew} then follows from \cite{Nu109} Proposition 4.5 and from the explicit formula $\hat{\Psi}\equiv0$ that holds in the case $q=0$.
\end{rmk}

Proposition \ref{ExplSolW} shows that in a Brownian framework the BSDE \eqref{BSDE} always admits a solution if $q\in[0,1)$. In view of  Corollary \ref{cornew}  this property extends to the general framework under the condition that $\lambda\cdot M$ is a BMO martingale. In particular, there is a unique bounded solution and it is given by the opportunity process for the utility maximization problem.

Let us now contrast this with the situation when $q<0$. The example in Subsection \ref{SecExpMom} provides a bounded BMO martingale $\lambda\cdot M$ such that the corresponding BSDE admits no solution. For this example the utility maximization problem satisfies $u(1)=+\infty$ and is thus degenerate ($\Psi_0\equiv+\infty$). 

The question now becomes whether, given an arbitrary $\lambda$ such that $\lambda \cdot M$ is a BMO martingale and $q<0$, we can still guarantee a bounded solution $\Psi$ to the BSDE \eqref{BSDE} when the utility maximization problem is \emph{nondegenerate}. We settle this question negatively in the next section providing an example for which Assumption \ref{ass_expmom} as well as the BMO property of $\lambda\cdot M$ hold, but the BSDE \eqref{BSDE} does not have a bounded solution.

To counterbalance this negative result in Section \ref{sec6} we provide, via the John-Nirenberg inequality, a condition on the order of the \emph{dynamic} exponential moments of the mean-variance tradeoff that guarantees boundedness of $\hat{\Psi}$. This is accompanied by a further example showing that this condition cannot be improved. To conclude, Corollary \ref{cornew}  and Theorem \ref{dem} provide a full characterization of the boundedness of solutions to the BSDE \eqref{BSDE} in terms of the dynamic exponential moments of $\lo \lambda \cdot M,\lambda\cdot M\ro$ for a BMO martingale $\lambda\cdot M$.

%%%%%%%%%%%%%%%%%%%%%%%%%%%%%%%%%%%%%%%%%%%%%%%%%%%%%%%%%%%%%%%%%%%%
%% Counterexample
%%%%%%%%%%%%%%%%%%%%%%%%%%%%%%%%%%%%%%%%%%%%%%%%%%%%%%%%%%%%%%%%%%

\section{Counterexamples to the Boundedness of BSDE Solutions}\label{SecCountBound}
We know that an optimal pair for the utility maximization problem gives rise to a triple $(\hat{\Psi},\hat{Z},\hat{N})$ solving the BSDE \eqref{BSDE}. Conversely, under suitable conditions, BSDE theory, based on \cite{BH08}, \cite{K00} or the results stated in the appendix, provides solutions to the BSDE with $\hat{\Psi}$ bounded (in $\mathfrak{E}$), with uniqueness in the class of bounded processes (in $\mathfrak{E}$). We now present an example of a BMO martingale $\lambda\cdot M$ which satisfies Assumption \ref{ass_expmom} and for which the BSDE \eqref{BSDE} related to the utility maximization problem has an unbounded solution for a given $p$. 

We develop this example in three steps. Firstly, we show that Assumption \ref{ass_expmom} alone (rather unsurprisingly) is not sufficient to guarantee a bounded BSDE solution. The corresponding $\lambda\cdot M$ involved is however not a BMO martingale. The second example is of BMO type, but lacks finite exponential moments of a sufficiently high order. It resembles the example provided in Subsection \ref{SecExpMom}. Finally, we combine these two examples to construct a BMO martingale $\lambda\cdot M$ such that $\lo \lambda\cdot M,\lambda\cdot M\ro_T$ has all exponential moments, but for which the BSDE does not allow for a bounded solution. Although this last step leaves the first two obsolete, we believe that the outlined presentation helps the reader in gaining insight into the nature of the degeneracy. In addition it hints at the minimal sufficient condition in Theorem \ref{dem} below. Namely, instead of simply requiring both the BMO and the exponential moments properties, they should be combined into a dynamic condition. While there is no reason to have a bounded solution if one requires only the BMO and the exponential moments properties, e.g. see the a priori estimate for general BSDEs in \cite{MW10} Proposition 3.1, constructing counterexamples appears to be nontrivial and similar ideas will be used to show sharpness of the dynamic condition in Section \ref{sec6}. Since in the present section we  construct suitable counterexamples, let $M=W$ be again a one-dimensional $\mathbb{P}$-Brownian motion in its augmented natural filtration.  

\subsection{Unbounded Solutions under All Exponential Moments}\label{CountBound1}
Let us assume the market price of risk is given by $\lambda:=-\sgn(W)\sqrt{|W|}$ so that the stock price dynamics read as follows, \[\frac{dS_t}{S_t}=dW_t-\sgn(W_t)\sqrt{|W_t|}\,dt.\] Note that in the above definition ``$-\sgn$" is motivated by economic rationale, to simulate a certain reverting behaviour of the returns. Assumption \ref{ass_expmom} is satisfied since \[\int_0^T\lambda^2_t\,d\lo M,M\ro_t=\int_0^T |W_t|\,dt\leq T\cdot\sup_{\substack{0\leq t\leq T}}|W_t|\]and by Doob's inequality, for $\varrho>1$,
\begin{align*}
\E\!\left[\exp\!\left( \varrho \sup_{\substack{0\leq t\leq T}}|W_t|\right)\right]&= \E\!\left[ \sup_{\substack{0\leq t\leq T}}\exp\!\left(\varrho|W_t|\right)\right]\leq \left(\frac{\varrho}{\varrho-1}\right)^\varrho \E\!\left[\exp\!\left(\varrho| W_T|\right)\right]\\&\leq 2\left(\frac{\varrho}{\varrho-1}\right)^\varrho e^{\,\varrho^2\,T/2}
<+\infty.
\end{align*}
Now let $p\in(0,1)$ so that $q<0$ and let $(\hat{X},\hat{Y})$ be the optimizers of the utility maximization problem, where $\hat{X}_0=x>0$ and $\hat{Y}_0=y:=u'(x)$. Since we are in a complete Brownian framework we have that $\hat{Y}=y\,\mc{E}(\sgn(W)\sqrt{|W|}\cdot W)$. If $\hat{\nu}$ denotes the optimal investment strategy we derive from Theorem \ref{thmuniq} that $(\hat{\Psi},\hat{Z},0)$ is the unique solution to the BSDE \eqref{BSDE} where $\hat{\Psi}:=\log(\hat{Y}\!/\,U'(\hat{X}))\in\mathfrak{E}$ and $\hat{Z}:=\sgn(W)\sqrt{|W|}+(1-p)\hat{\nu}$. According to \cite{Nu109} Proposition 4.5 $\hat{\Psi}$ is bounded if and only if $\hat{Y}$ satisfies the reverse H\"{o}lder inequality 
\begin{equation} \label{reversHolder} 
\E\!\left[\left(\hat{Y}_T\big/\hat{Y}_\tau\right)^q\bigg|\,\mc{F}_\tau\right]\leq c_{\textit{\tiny{rH}},p} 
\end{equation}
for some positive constant $c_{\textit{\tiny{rH}},p}$ and all stopping times $\tau$ valued in $[0,T]$, which we show is not the case.

The family $\{|W_t|\,|\,t\in[0,T]\}$ is uniformly integrable since $\E\bigl[W_t^2\bigr]=t\leq T$, so we may apply the stochastic Fubini theorem (\cite{Be06} Lemma A.1) to get, for some $t\in(0,T)$, via Jensen's inequality,
\begin{align*}
\E\!\left[\left(\hat{Y}_T\big/\hat{Y}_t\right)^q\bigg|\,\mc{F}_t\right]&=\E\!\left[\exp\!\left( q\int_t^T\sgn(W_s)\sqrt{|W_s|}\,dW_s-\frac{q}{2}\int_t^T|W_s|\,ds\right)\Bigg|\,\mc{F}_t\right]\\&\geq \exp \!\left(\E\!\left[ q\int_t^T\sgn(W_s)\sqrt{|W_s|}\,dW_s-\frac{q}{2}\int_t^T|W_s|\,ds\,\bigg|\,\mc{F}_t\right]\right)\\&=\exp\!\left(-\frac{q}{2} \,\E\!\left[ \int_t^T|W_s|\,ds\,\bigg|\,\mc{F}_t\right]\right)=\exp\!\left(-\frac{q}{2} \int_t^T\E\!\left[|W_s|\,\big|\,\mc{F}_t\right]\,ds\right)\\&\geq \exp\!\left(-\frac{q}{2}\int_t^T|W_t|\,ds\right)=\exp\!\left(-\frac{q(T-t)}{2}\,|W_t|\right).
\end{align*}
Since the last random variable is unbounded it cannot be the case that \eqref{reversHolder} holds, hence $\hat{\Psi}$ cannot be bounded.

However, $\lambda\cdot W$ from this example is not a BMO martingale since for $t\in (0,T)$,
\[\E\!\left[ \int_t^T|W_s|\,ds \, \bigg|\,\mc{F}_t\right]= \int_t^T\E\!\left[|W_s|\,\big|\,\mc{F}_t\right]ds  \geq  (T-t) \,|W_t|,\]
which shows that $\|\lambda\cdot M\|_{\text{BMO}_2}$ cannot be finite.

\subsection{Unbounded Solutions under the BMO Property}\label{CountBound2}
We continue with a BMO example for which the solution to the BSDE \eqref{BSDE} is unbounded. The idea is the following, from Proposition \ref{propnosol}, for $q<0$, there exists $\lambda$ with $ \lambda\cdot W$ a BMO martingale such that the BSDE \eqref{BSDE} has no solution (in any class of possible solutions). Replacing this $\lambda$ by $c \,\lambda$ for a constant $c$, it follows from \eqref{condex} below that the BSDE has either no solution (for $|c| \geq 1$) or has a solution which is bounded and fulfills a BMO property (for $|c| < 1$). This dichotomy is in line with the fact that for a BMO martingale $\overline{M}$ the set of all $q < 0$ such that $\mathcal{E}(\overline{M})$  satisfies the reverse H\"{o}lder inequality $R_{q}$ is open; compare Lemma \ref{lemholder} below. The insight then is to make $c$ a \emph{random} variable in order to construct $\lambda$ such that the BSDE \eqref{BSDE} has a solution which is not bounded. More precisely, we have the following result.
\begin{prop}\label{ThmCountBound2} For every $q < 0$ there exists a $\lambda$ with $ \lambda  \cdot W $ a BMO martingale such that,
\begin{enumerate}
  \item The BSDE \eqref{BSDE} has a unique solution $(\hat{\Psi},\hat{Z},\hat{N}\equiv0)$ with $\mc{E}\bigl( [(1-q)\hat{Z}-q\lambda] \cdot W\bigr)$ a martingale. In particular, solving \eqref{BSDE} and setting $(X,Y)$ as suggested by Theorem \ref{thmuniq} gives the pair of primal and dual optimizers.
  \item There does \emph{not} exist a solution $({\Psi},{Z})$ to \eqref{BSDE} with ${Z} \cdot W $ a BMO martingale or $\Psi$ \emph{bounded}.
\end{enumerate}
\end{prop}
\begin{proof}
For $t\in[0,T]$ we set
\begin{equation*} 
\lambda_t:=\frac{\pi\alpha}{2\sqrt{-q(T-t)}}\,\mathbf{1}_{]\!]T/2,\tau]\!]}(t,\cdot),
\end{equation*}
where \[\alpha := \frac{2}{\pi} \arccos\sqrt{\Phi\!\left(\sqrt{2/T}\, W_{T/2}\right)}\] for $\Phi$ the standard normal cumulative distribution function and $\tau$ the stopping time from the proof of Proposition \ref{propnosol},
\begin{equation}\label{stoptau}
\tau := \inf\Bigg\{ t >\frac{T}{2}\,\Bigg|\, \Bigg| \int_{T/2}^{t} \frac{1}{\sqrt{T-s}}\,dW_{s}\Bigg|\geq1 \Bigg\}.
\end{equation}
Note that $\Phi\bigl(\sqrt{2/T}\,W_{T/2}\bigr)$ is uniformly distributed on $(0,1]$ and that $\alpha$ is valued in $[0,1)$ a.s. It follows immediately that $\lambda\cdot W$ is bounded by $\frac{\pi}{2\sqrt{-q}}$, in particular it is a BMO martingale.

Using \cite{Ka94} Lemma 1.3 in the same way as in the proof of \cite{FdR10} Lemma A.1, we obtain that
\begin{align*}
\E\Bigl[ \mc{E}(- \lambda \cdot W)_T^q \Bigr]&\leq e^{\frac{\pi\sqrt{-q}}{2}}\,\E\Biggl[\exp\Biggl(\frac{\pi^2\alpha^2}{8} \int_{T/2}^{\tau} \frac{1}{T-t}\,dt \Biggr)\Biggr]=e^{\frac{\pi\sqrt{-q}}{2}}\,\E\Biggl[\frac{1}{\cos\bigl(\pi\alpha/2\bigr)}\Biggr]\\&= e^{\frac{\pi\sqrt{-q}}{2}}\,\E\Biggl[\frac{1}{\sqrt{\Phi \big(\sqrt{2/T}\, {W}_{T/2}\big)}} \Biggr]=2 e^{\frac{\pi\sqrt{-q}}{2}}<+\infty,
\end{align*}
so that Proposition \ref{ExplSolW} gives the first assertion. Due to the $\mc{F}_{T/2}$-measurability of $\alpha$ and the $\mc{F}_{T/2}$-independence of $\tau$, we have
\begin{align*}
\exp\bigl((1-q)\hat{\Psi}_{T/2}\bigr)&\geq e^{\frac{-\pi\sqrt{-q}}{2}}\,\E\Biggl[\exp\Biggl( \frac{\pi^2\alpha^2}{8}\int_{T/2}^{\tau} \frac{1}{T-t}\,dt \Biggr) \Bigg|\,\mc{F}_{T/2}\Biggr]\\&= e^{\frac{-\pi\sqrt{-q}}{2}}\,\E\Biggl[\exp\Biggl( \frac{\pi^2x^2}{8}\int_{T/2}^{\tau} \frac{1}{T-t}\,dt \Biggr)\Biggr]\Bigg|_{x=\alpha}=\frac{e^{\frac{-\pi\sqrt{-q}}{2}}}{\cos(\pi\alpha/2)}.
\end{align*}
This shows that
\[(1-q)\hat{\Psi}_{T/2}\geq-\frac{\pi\sqrt{-q}}{2}-\frac12\,\log\Bigl(\Phi\Bigl(\sqrt{2/T}\,W_{T/2}\Bigr)\Bigr),\]
which is unbounded by the uniform distribution of $\Phi\bigl(\sqrt{2/T}\,W_{T/2}\bigr)$.

For item (ii) assume that there exists a solution $({\Psi},Z)$ to \eqref{BSDE} with $ Z \cdot W$ a BMO martingale or ${\Psi}$ bounded. By Theorem \ref{solbdd}(ii) we can restrict ourselves to assuming that $ \Psi$ is bounded, which implies that $ Z \cdot W$ is a BMO martingale so that $\mc{E}\bigl( [(1-q){Z}-q\lambda] \cdot W\bigr)$ is a martingale. By uniqueness, $({\Psi},Z)$ coincides with $(\hat{\Psi}, \hat{Z})$ in contradiction to the unboundedness of $\hat{\Psi}$.
\end{proof}

\subsection{Unbounded Solutions under All Exponential Moments \emph{and} the BMO Property}\label{CountBound3}
The two previous subsections raise the question whether we can find a BMO martingale $ \lambda \cdot M$ such that its quadratic variation has all exponential moments and the BSDE \eqref{BSDE} has only an unbounded solution. Roughly speaking, the idea is to combine the above two examples by translating the crucial distributional properties of $|W|$ and $\alpha$ into the corresponding properties of a suitable stopping time $\sigma$. This guarantees that the BMO property and the exponential moments condition are satisfied simultaneously, while we can also achieve the unboundedness of the BSDE solution by using independence. Table~\ref{tab} summarizes the key properties.
\begin{table}[h!]
  \centering 
\begin{tabular}{lll} \toprule[0.12 em]
&  \textbf{Form of $\lambda^{2}_{t}$} & \textbf{Crucial properties} \\
\midrule[0.12 em]
 \multirow{2}{*}{\textbf{First example (see \ref{CountBound1})}} & \multirow{2}{*}{$|W_{t}|$} & $|W_{t}|$ is unbounded,\\ && has all exponential moments \\ \midrule
 \multirow{2}{*}{\textbf{Second example (see \ref{CountBound2})}}  & \multirow{2}{*}{$\frac{\pi^2\alpha^{2}}{4(-q)} \frac{1}{T-t}\,  \mathbf{1}_{]\!]T/2,\tau]\!]}(t,\cdot)$} & $\alpha^{2} \in [0,1)$, $\mathbb{P}(\alpha^{2}\geq\varrho) >0\,\forall \varrho < 1$, \\
 & & $\E\big[1/{\cos(\alpha \pi /2) } \big] < +\infty$  \\
\midrule[0.12 em]
 \multirow{2}{*}{\textbf{Combination}} &  \multirow{2}{*}{$\frac{\pi^2}{4(-q)} \frac{1}{T-t} \mathbf{1}_{]\!]T/2,\tau\wedge\sigma]\!]}(t,\cdot) $} & $ \sigma\in(T/2,T]$, $\mathbb{P}(\sigma\geq\varrho) >0\,\forall \varrho < T$,\\& & $\int_0^\sigma\frac{1}{T-t}\,dt$ has all exponential moments \\
\bottomrule[0.12 em]
\end{tabular}
\vspace{2mm}
  \caption{Comparison of the BSDE examples from Section \ref{SecCountBound}} \label{tab} 
\end{table}
\begin{thm}\label{ThmCountBound3} For every $q < 0$, there exists a $\lambda$ such that,
\begin{enumerate}
\item The process $ \lambda  \cdot W $ is a BMO martingale. 
  \item For all $\varrho>0$ we have $\E\bigr[\exp\bigl(\varrho\int_0^T\lambda^2_t\,dt\bigr)\bigr]<+\infty$.
  \item The BSDE \eqref{BSDE} has a unique solution $(\hat{\Psi},\hat{Z},\hat{N}\equiv0)$ with $\mc{E}\bigl( [(1-q)\hat{Z}-q\lambda] \cdot W\bigr)$ a martingale. In particular, solving \eqref{BSDE} and setting $(X,Y)$ as suggested by Theorem \ref{thmuniq} gives the pair of primal and dual optimizer.
  \item There does \emph{not} exist a solution $({\Psi},{Z})$ to \eqref{BSDE} with ${Z} \cdot W $ a BMO martingale or $\Psi$ \emph{bounded}.
\end{enumerate}
\end{thm}
\begin{proof}
Let us first construct $\sigma$ with the desired distributional properties. We define the nonnegative continuous function $f:(T/2,T]\to \mathbb{R}$, $f(s):= c_0\cdot e^{-\frac{1}{T-s}}$, where $c_0>0$ is a constant such that $\int_{T/2}^Tf(s)\,ds=1.$ We then consider the strictly increasing function $F:(T/2,T]\to (0,1]$, $F(s):=\int_{T/2}^sf(u)\,du$ and its inverse $F^{-1}:(0,1]\to(T/2,T]$. We set \[\sigma:=\bigl(F^{-1}\circ\Phi\bigr)\Bigl(\sqrt{2/T}\,W_{T/2}\Bigr),\] so that $\sigma$ is an $\mc{F}_{T/2}$-measurable random variable with values in $(T/2,T]$ and cumulative distribution function $F$. Now define for $t\in[0,T]$,
\begin{equation*}
\lambda_t:=\frac{\pi}{2\sqrt{-q(T-t)}}\,\mathbf{1}_{]\!]T/2,\tau\wedge \sigma]\!]}(t,\cdot),
\end{equation*}
where $\tau$ is the stopping time from \eqref{stoptau}. It follows immediately that $\lambda\cdot W$ is bounded by $\frac{\pi}{2\sqrt{-q}}$, hence a BMO martingale.

Let us now show that $\int_0^T\lambda^2_t\,dt$ has all exponential moments. Take $\varrho>0$ and $\bar{\varrho}\geq\frac{\pi^2\varrho}{4(-q)}\vee 2$ an integer. We derive
\begin{align*}
\E\Biggl[\exp\Biggl( \varrho \int_{0}^{T}  \lambda_{t}^{2}\, d t \Biggr) \Biggr]
&\leq \E\Biggl[\exp\Biggl( \bar{\varrho} \int_{T/2}^{\sigma}  \frac{1}{T-t}\, d t \Biggr) \Biggr]
\\&=\E\bigl[\exp\bigl( \bar{\varrho}\,[ \log(T/2)-\log(T-\sigma)]\bigr) \bigr]=(T/2)^{\bar{\varrho}}\,\E\Biggl[\frac{1}{(T-\sigma)^{\bar{\varrho}}}\Biggr]
\\&=c_0(T/2)^{\bar{\varrho}}\int_{T/2}^T\biggl(\frac{1}{T-s}\biggr)^{\bar{\varrho}}e^{-\frac{1}{T-s}}\,ds=c_0(T/2)^{\bar{\varrho}}\int_{2/T}^{+\infty}u^{\bar{\varrho}-2}e^{-u}\,du\\
&=c_0(T/2)^{\bar{\varrho}}(\bar{\varrho}-2)!\,e^{-2/T}\sum_{k=0}^{\bar{\varrho}-2}\frac{(2/T)^k}{k!}<+\infty,
\end{align*}
where in the last equality we used the representation of the incomplete gamma function at integer points (or, directly, integration by parts). A standard argument then shows that $\E\bigl[ \mc{E}(- \lambda \cdot W)_T^q \bigr] < + \infty,$ see the proof of Theorem \ref{dem}(i) below. It follows from the Proposition \ref{ExplSolW} that there exists a unique solution $(\hat{\Psi},\hat{Z})$ to the BSDE \eqref{BSDE} such that $\mc{E}([(1-q)\hat{Z}-q\lambda ] \cdot W)$ is a martingale and the first component is given by
\begin{equation*}
\hat{\Psi}_t = \frac{1}{1-q}\, \log\Bigl( \E\bigl[ \mc{E}(- \lambda \cdot W)_{t,T}^q \,\big|\, \mathcal{F}_t\bigr]\Bigr),\quad t\in[0,T],\quad \text{a.s.}
\end{equation*}
We deduce that a.s.
\begin{align*}
\exp\bigl((1-q)\hat{\Psi}_{T/2}\bigr)&\geq e^{\frac{-\pi\sqrt{-q}}{2}}\E\Biggl[\exp\Biggl( \frac{\pi^2}{8}\int_{T/2}^{\tau\wedge\sigma} \frac{1}{T-t}\,dt \Biggr) \Bigg|\,\mc{F}_{T/2}\Biggr]\\&= e^{\frac{-\pi\sqrt{-q}}{2}}\E\Biggl[\exp\Biggl( \frac{\pi^2}{8}\int_{T/2}^{\tau\wedge s} \frac{1}{T-t}\,dt \Biggr)\Biggr]\Bigg|_{s=\sigma}
\end{align*}
because $\sigma$ is $\mc{F}_{T/2}$-measurable and $\tau$ is independent from $\mc{F}_{T/2}$. From monotone convergence, it follows that
\begin{equation}\label{liminf}
\lim_{s\uparrow T} \E\Biggl[\exp\Biggl(\frac{\pi^2}{8} \int_{T/2}^{\tau\wedge s}\frac{1}{T-t}\,dt \Biggr) \Biggr] =
\E\Biggl[\exp\Biggl( \frac{\pi^2}{8} \int_{T/2}^{\tau}  \frac{1}{T-t} \,dt \Biggr) \Biggr] = +\infty.
\end{equation}
We now fix $K >0$ and take $s_{0}\in(T/2,T)$ such that 
\[\E\Biggl[\exp\Biggl(\frac{\pi^2}{8}\int_{T/2}^{\tau\wedge s} \frac{1}{T-t}\,dt \Biggr) \Biggr] \geq e^{(1-q)K+\frac{\pi\sqrt{-q}}{2}} \quad \textrm{for all }s \in [s_{0},T),\]
which is possible by \eqref{liminf}. This implies ${\mathbb{P}}(\hat{\Psi}_{T/2} \geq K) \geq \mathbb{P}(\sigma\geq s_{0})=1-F(s_0)>0$ since $s_0<T$, in particular $\hat{\Psi}$ is unbounded. The last item then follows as in the previous proof.
\end{proof}
\begin{rmk}
It is interesting to compare, for different constants $c \in\mathbb{R}$, the above different definitions of ${\lambda}$ regarding the behaviour of the solution to the BSDE 
\begin{equation} \label{cbsde}
d{\Psi}_{t} = {Z}_{t}\, d{W}_{t}+\frac{q}{2}\,({Z}_{t}+c\lambda_t)^{2}\, dt -\frac12\,Z_t^2\,dt, \quad {\Psi}_{T} = 0.
\end{equation}
In the example of Proposition \ref{propnosol} $\lambda_t^2$ is of the form $\frac{\pi^2}{4(-q)}\frac{1}{T-t}\, \mathbf{1}_{]\!]T/2,\tau]\!]}(t,\cdot)$, while in Subsection \ref{CountBound2} $\lambda^{2}_{t}$ equals $\frac{\pi^2\alpha^{2}}{4(-q)} \frac{1}{T-t}\,  \mathbf{1}_{]\!]T/2,\tau]\!]}(t,\cdot)$, which we modified to $\frac{\pi^2}{4(-q)} \frac{1}{T-t}\, \mathbf{1}_{]\!]T/2,\tau\wedge\sigma]\!]}(t,\cdot) $ in the above discussion. Table \ref{tab2} shows that by introducing additional random variables in the construction of $\lambda$, the BSDE \eqref{cbsde} becomes solvable for bigger values of $|c|$, but the solution for $|c| \geq 1$ is unbounded. The assertions of Table \ref{tab2} can be deduced from the arguments in the above proofs together with the additional calculation given in \eqref{bddcos} below.
 
\begin{table}[h!]
  \centering 
\begin{tabular}{llccc} \toprule[0.12 em]
& \multirow{2}{*}{\textbf{Form of $\lambda^{2}_{t} \Bigl(\text{up to }\frac{\pi^2}{4(-q)}\Bigr)$}} &\multicolumn{3}{c}{\textbf{Solution to the BSDE \eqref{cbsde}}}\\
   & &$ |c| \in [0,1)$ & $|c| = 1$ & $|c| > 1$ \\
\midrule[0.12 em]
 \textbf{Example from \ref{SecExpMom}} & $  \frac{1}{T-t} \,\mathbf{1}_{]\!]T/2,\tau]\!]}(t,\cdot) $& bounded &\multicolumn{2}{c}{no solution}  \\ \midrule[0.1 em]
 \textbf{Example from \ref{CountBound2}}   & $\alpha^{2}  \frac{1}{T-t}\, \mathbf{1}_{]\!]T/2,\tau]\!]} (t,\cdot)$  & bounded & unbounded & no solution  \\
\midrule[0.1 em]
\textbf{Example from \ref{CountBound3}} & $ \frac{1}{T-t}\,  \mathbf{1}_{]\!]T/2,\tau\wedge\sigma]\!]}(t,\cdot) $ & bounded &\multicolumn{2}{c}{unbounded}\\\bottomrule[0.12 em]
\end{tabular}
\vspace{2mm}
  \caption{Description of solutions to the BSDE \eqref{cbsde}} \label{tab2} 
\end{table}
\end{rmk}
\begin{rmk}
In the Subsections \ref{SecExpMom}, \ref{CountBound2} and above we constructed several counterexamples to the boundedness of BSDE solutions. Such examples can also be given in Markovian form using Az\'{e}ma-Yor martingales. More precisely, for $t\in[0,T)$ let \[X_t:=\int_0^t\frac{1}{\sqrt{T-s}}\,\mathbf{1}_{(T/2,T]}(s)\,dW_s\] and $\tau$ as in (\ref{stoptau}), noting that $\tau=\inf\bigl\{ t \geq0\,\big|\, |X_t|\geq1 \bigr\}$. If $\underline{X}$ ($\ol{X}$) denotes the running minimum (maximum) of $X$, then \[U:=-X\underline{X}+\frac{1}{2}\,\underline{X}^2\quad\text{ and }\quad V:=X\ol{X}-\frac12\,\ol{X}^2 \] are continuous local martingales on $[0,T)$ by Az\'{e}ma and Yor \cite{AY79}. To close the continuity gap at $t=T$ consider $\check{\tau}:=\inf\bigl\{ t \geq0\,\big|\, |X_t|\geq2 \bigr\}$ and set $M:=(X^{\check{\tau}},U^{\check{\tau}},V^{\check{\tau}})$. A calculation then shows that \[\mathbf{1}_{[\![0,\tau]\!]}= \mathbf{1}_{\{g_1(M^1,M^2)\geq-1\}}\mathbf{1}_{\{g_2(M^1,M^3)\leq 1\}}\]for \[g_1(x,u):=\bigl(x-\sqrt{x^2+2u}\bigr)\mathbf{1}_{\{x^2\geq-2u\}}\quad\text{ and }\quad g_2(x,v):=\bigl(x+\sqrt{x^2-2v}\bigr)\mathbf{1}_{\{x^2\geq2v\}}.\]
For the analogue of Proposition \ref{propnosol} we could now take the three-dimensional $M$ from above, but actually the one-dimensional local martingale $M^1$ turns out to be sufficient. We obtain that for every $q<0$ there exists a predictable process $\lambda$ which is a function of $M^1$ such that $\lambda\cdot M^1$ is a bounded martingale satisfying $\E\bigl[ \mc{E}(- \lambda \cdot M^1)_T^q \bigr] = + \infty.$ Indeed, $\lambda:=\frac{\pi}{2\sqrt{-q}}\,\mathbf{1}_{\{|M^1|\leq1\}}$ gives the claim. For the analogues of Proposition 5.1 and Theorem 5.2 we set $M^4:=W\mathbf{1}_{[0,T/2]}+W_{T/2}\mathbf{1}_{(T/2,T]}$ and $\lambda=(\lambda^1,0,0,0)$, where \[\lambda^1:=\frac{\pi\alpha}{2\sqrt{-q}}\,\mathbf{1}_{[\![0,\tau]\!]}\quad\text{ or }\quad\lambda^1:=\frac{\pi}{2\sqrt{-q}}\,\mathbf{1}_{[\![0,\tau\wedge \sigma]\!]},\] which again prove to be Markovian in $M$ and for which the statements of the cited results remain valid.
\end{rmk}

%%%%%%%%%%%%%%%%%%%%%%%%%%%%%%%%%%%%%%%%%%%%%%%%%%%%%%%%%%%%%%%%%%%%%%%%%%%%
%% Characterization of Boundedness
%%%%%%%%%%%%%%%%%%%%%%%%%%%%%%%%%%%%%%%%%%%%%%%%%%%%%%%%%%%%%%%%%%%%%%%%%%%%%%

\section{Characterization of Boundedness of BSDE Solutions}\label{sec6}
We have already shown that for a BMO martingale $\lambda\cdot M$ and $q\in[0,1)$ the BSDE \eqref{BSDE} allows for a unique bounded solution. In the previous section we gave some examples to show that for $q<0$ the situation is different. In this section we complete the analysis by developing a sufficient condition that guarantees (necessarily unique) bounded solutions to \eqref{BSDE}. It is also shown that this particular condition cannot be improved. More precisely, we consider here a more general situation where $(\mc{F}_t)_{t\in[0,T]}$ is not necessarily a continuous filtration, but only a filtration satisfying the usual conditions. We assume that the local martingale $M$ is still continuous. In this case the BSDE \eqref{BSDE} is replaced by 
\begin{multline}
d\Psi_t= Z_t^\tr\,dM_t + dN_t -\frac{1}{2}\,d\lo N^{c},N^{c}\ro_t+ \log(1+\Delta N_{t}) - \Delta N_{t}\label{BSDE2} \\+\frac{q}{2}(Z_t+\lambda_t)^\tr\,d\lo
M,M\ro_t(Z_t+\lambda_t)-\frac{1}{2}\,Z_t^\tr\,d\lo M,M\ro_tZ_t,\quad\Psi_T=0.
\end{multline}
We mention that all the results which depend only on the specific continuous local martingale $M$ also hold in this more general setting. In particular, the statements of \cite{Nu109} Proposition 4.5 and our Corollary \ref{cornew} continue to hold for the BSDE \eqref{BSDE2} in place of the BSDE \eqref{BSDE}.

\subsection{The critical exponent of a BMO martingale}

We will see that the boundedness of BSDE solutions depends crucially on the so-called critical exponent of the market price of risk. After defining the critical exponent of a general BMO martingale, we give some properties which will be exploited later. We then explain for a general BSDE how the critical exponent is related to boundedness. In addition to the counterexamples in Section \ref{SecCountBound}, this gives a  motivation for our main result, Theorem  \ref{dem}, about how to characterize bounded solutions.

We recall the John-Nirenberg inequality for the convenience of the reader. In what follows $\ol{M}$ is an arbitrary continuous martingale on $[0,T]$ with $\ol{M}_0=0$.
\begin{lem}[Kazamaki \cite{Ka94} Theorem 2.2] \label{KaJohnNirenL}If $\big\|\ol{M}\big\|_{\mathrm{BMO}_2}<1$ then for every stopping time $\tau$ valued in $[0,T]$
\begin{equation}\label{KaJohnNiren}
\E\Big[\exp\Bigl(\lo \ol{M},\ol{M}\ro_T-\lo \ol{M},\ol{M}\ro_\tau\Bigr)\Big|\,\mc{F}_\tau\Big]\leq\frac{1}{1-\big\|\ol{M}\big\|_{\mathrm{BMO}_2}^2}.
\end{equation}
\end{lem}
Using the definition of \cite{Ka94} and the terminology of Schachermayer \cite{Sch96}, we define the \emph{critical exponent} $b$ via
\begin{equation}\label{defb}
b(\ol{M}):=\sup\!\left\{b\geq0\,\bigg|\,\sup_{\tau}\Big\|\E\Big[\exp\!\Big(b\bigl(\lo \ol{M},\ol{M}\ro_T-\lo \ol{M},\ol{M}\ro_\tau\bigr)\Big)\,\Big|\,\mc{F}_\tau\Big]\Big\|_{L^\infty}<+\infty\right\},
\end{equation}
where the supremum inside the brackets is over all stopping times $\tau$ valued in $[0,T]$. We refer to this inner supremum as a \emph{dynamic exponential moment} of $\langle \ol{M},\ol{M}\rangle$ of order $b$. A consequence of Lemma \ref{KaJohnNirenL} is then that a martingale $\ol{M}$ is a BMO martingale if and only if $b(\ol{M})>0$. In addition, the following lemma shows that the supremum in \eqref{defb} is never attained.
\begin{lem} \label{lemb}
Let $k> 0$ and $\ol{M}$ be a continuous martingale with
\begin{equation} \label{condlemb}
\sup_{\tau}\Big\|\E\Big[\exp\!\Big(k\bigl(\lo \ol{M},\ol{M}\ro_T-\lo \ol{M},\ol{M}\ro_\tau\bigr)\Big)\,\Big|\,\mc{F}_\tau\Big]\Big\|_{L^\infty}<+\infty.
\end{equation}
Then there exists $\tilde{k} > k$ such that 
\[\sup_{\tau}\Big\|\E\Big[\exp\!\Big(\tilde{k}\bigl(\lo \ol{M},\ol{M}\ro_T-\lo \ol{M},\ol{M}\ro_\tau\bigr)\Big)\,\Big|\,\mc{F}_\tau\Big]\Big\|_{L^\infty}<+\infty\]
and hence $b\bigl(\ol{M}\bigr) > k$.

\end{lem}
\begin{proof}
Inspired by \cite{Ka94} Corollary 3.2 we aim to apply Gehring's inequality. To this end, fix a stopping time $\tau$ and set $\Gamma_{t} := \exp\!\Big(k\bigl(\lo \ol{M},\ol{M}\ro_t-\lo \ol{M},\ol{M}\ro_\tau\bigr)\Big)$ for $t \in [0,T]$. For each $\mu > 1$, we then define the stopping time $\tau_{\mu} := \inf\{t\in[\![\tau,T]\!]\,|\,\Gamma_{t} > \mu \}$. It follows from $\mu > 1$ and the continuity of $\Gamma$ that
\begin{equation} \label{subsetb}
\Gamma_{\tau_{\mu}} = \mu\quad\textrm{on}\quad\{\tau_{\mu} < +\infty\}.
\end{equation}
Since $\Gamma$ is nondecreasing, we have that $\{ \Gamma_{T} > \mu\} = \{ \tau_{\mu} < +\infty \}$ and this event is $\mc{F}_{\tau_{\mu}}$-measurable. Therefore, we obtain
\begin{align*}
\E\bigl[ \mathbf{1}_{\{ \Gamma_{T} > \mu\}} \Gamma_{T}  \bigr]  & = \E\bigl[ \mathbf{1}_{\{ \Gamma_{T} > \mu\}}  \E[ \Gamma_{T}| \mc{F}_{\tau_{\mu}}]  \bigr]\\  &  =  \E\biggl[ \mathbf{1}_{\{ \Gamma_{T} > \mu\} } \Gamma_{\tau_{\mu}} \E\Big[ \exp\!\Big(k\bigl(\lo \ol{M},\ol{M}\ro_{T }-\lo \ol{M},\ol{M}\ro_{\tau_{\mu} }\bigr)\Big)\Big| \mc{F}_{\tau_{\mu}}\Big]  \biggr] \\ & \leq c_k\, \E\bigl[ \mathbf{1}_{\{ \Gamma_{T} > \mu\} } \Gamma_{\tau_{\mu}} \bigr],
\end{align*}
where we used \eqref{condlemb} and denoted its left-hand side by $c_k$. Fix now $ \eps \in (0,1)$. Using \eqref{subsetb}, we derive
\[\E\bigl[ \mathbf{1}_{\{ \Gamma_{T} > \mu\} } \Gamma_{\tau_{\mu}}\bigr] = \mu^{\eps}  \E\bigl[ \mathbf{1}_{ \{ \Gamma_{T} >   \mu\}} \Gamma_{\tau_{\mu}}^{1- \eps} \bigr]\leq \mu^{\eps}  \E\bigl[ \mathbf{1}_{ \{ \Gamma_{T} >  \mu\}}  \Gamma_{T}^{1- \eps}\bigr]\]
and conclude that \[\E\bigl[ \mathbf{1}_{\{ \Gamma_{T} > \mu\}} \Gamma_{T}  \bigr] \leq c_k\, \mu^{\eps}\,  \E\bigl[ \mathbf{1}_{ \{ \Gamma_{T} >  \mu\}} \Gamma_{T}^{1- \eps} \bigr].\]
It follows from the probabilistic version of Gehring's inequality given in \cite{Ka94} Theorem 3.5, however see Remark \ref{rmkGehring} below, that there exist $r > 1$ and $C > 0$ (depending on $\eps$ and $c_k$ only) such that
\[\E\bigl[  \Gamma_{T}^{r}  \bigr] \leq C \, \E  [  \Gamma_{T}]^{r}.\]
To obtain the conditional version, we take $A \in \mc{F}_{\tau}$ and derive from the same argument and Jensen's inequality that
\[\E\bigl[  \Gamma_{T}^{r}\mathbf{1}_{A}  \bigr] \leq C \, \E  [  \Gamma_{T}\mathbf{1}_{A}  ]^{r} \leq C \, \E  \bigl[  \E [\Gamma_{T}| \mc{F}_{\tau}]^{r}  \mathbf{1}_{A}  \bigr] \]
so that \[\E\bigl[  \Gamma_{T}^{r} | \mc{F}_{\tau}] \leq C \, \E  [  \Gamma_{T}| \mc{F}_{\tau}]^{r}\leq  c_k^{r}\,C \quad \textrm{a.s}.\]
Since this holds for any stopping time $\tau$, we conclude the proof by setting $\tilde{k} = r k$.
\end{proof}

We present another auxiliary result that will be applied in the next subsection.

\begin{lem} \label{lemholder}
Let $q < 0$ and $\ol{M}$ be a continuous BMO martingale such that the reverse H\"{o}lder inequality \eqref{reversHolder} holds for $ \mc{E}(\ol{M})$. Then there exists $\tilde{q} < q$ such that $ \mc{E}(\ol{M})$ satisfies the reverse H\"{o}lder inequality \eqref{reversHolder} with $\tilde{q}$.
\end{lem}
\begin{proof} We note that the reverse H\"{o}lder inequality $R_q$ for $q < 0$ is equivalent to the Muckenhoupt inequality $A_{\varrho}$ with $\varrho = 1 - 1/q > 1$, see \cite{Ka94} Definition 2.2. Therefore, the statement of Lemma \ref{lemholder} follows from \cite{Ka94} Corollary 3.3.
\end{proof}
\begin{rmk}\label{rmkGehring}
We mention that in the formulation of \cite{Ka94} Theorem 3.5 as well as the proof of \cite{Ka94} Corollary 3.2 there is a small gap which can be easily filled. Namely, for a nonnegative random variable $U$ and positive constants $K$, $\beta$ and $\eps\in(0,1)$ the author requires Gehring's condition
\begin{equation} \label{condgehring}
\E\bigl[ \mathbf{1}_{\{ U > \mu\}} U  \bigr] \leq K \mu^{\eps}\,  \E\bigl[ \mathbf{1}_{ \{ U  > \beta \mu\}} U^{1- \epsilon}  \bigr]
\end{equation}
to hold for \emph{all} $\mu > 0$, which cannot be satisfied for $U\in L^1$ unless $U=0$ a.s. This is because for an integrable $U \neq 0$ the right-hand side tends to zero as $\mu \downarrow 0$ whereas the left-hand side tends to $\E[U] > 0$. However, an inspection of the proof reveals that \eqref{condgehring} is needed only for $\mu > \E[U]$, i.e. \cite{Ka94} Theorem 3.5 should be stated for $\mu > \E[U]$ instead of $\mu > 0 $. If this is the case it then can be applied in the proof of \cite{Ka94} Corollary 3.2, where for $\mu > 0$ the following stopping time is considered, $\tau_{\mu} = \inf\{t\geq0\,|\, \mc{E}(\ol{M})_{t}^{p}> \mu\}$ for a continuous local martingale $\ol{M}$, see also the proof of Lemma \ref{lemb}. Then, the desired estimate $\mc{E}(\ol{M})_{\tau_{\mu}}^{p} \leq \mu$ is derived, but the latter holds for $\mu\geq1$ only, since for $\mu\in(0,1)$ we obtain that $\tau_\mu=0$ which in turn gives $\mc{E}(\ol{M})_{\tau_{\mu}}^{p}=\mc{E}(\ol{M})_{0}^{p}=1>\mu$.
\end{rmk}
 To illustrate how dynamic exponential moments lead to boundedness and help motivate our next theorem, consider the following BSDE in a continuous filtration
\begin{equation} \label{BSDEF3}
d\Psi_t= Z_t^\tr\,dM_t + dN_t -F(t,Z_t)\,dA_t-\frac{1}{2}\,d\lo N,N\ro_t ,\quad \Psi_T=\xi;
\end{equation}
where $\xi$ is a bounded random variable. We assume that the driver  $F$ is continuous and satisfies
\begin{align*}
|F(t,   z)| & \leq    \| B_{t}\eta_{t}\|^{2}+\frac{\gamma}{2} \| B_{t}z\|^{2},\\
 |F(t,   z_{1}) - F(t,z_{2})|  &    \leq  \beta_F \bigl(  \| B_{t}\eta_{t}\| +  \| B_{t}z_{1}\| +  \| B_{t}z_{2}\| \big)  \| B_{t}(z_{1}-z_{2})\|
\end{align*}
 for all $t \in [0,T]$ and $z,z_{1},z_{2} \in \mathbb{R}^{d}$, where $\beta_F, \gamma$ are constants and $\eta\cdot M$ is a BMO martingale such that $\E \big[\exp\big(  \tilde{\gamma}  \int_{0}^{T}    \| B_t \eta_{t} \|^2\, d A_{t}  \big) \big] <+ \infty$ for  $\tilde{\gamma} :=  \max\{1,\gamma\}$.

From \cite{MW10} Theorem 4.1 (noting that convexity of $F$ in $z$ is not needed by \cite{MW10} Remark~4.3), we obtain that there exists a solution $(\Psi,Z,N)$ to \eqref{BSDEF3} which satisfies
\[
|\Psi_{t}|    \leq    \| \xi \|_{L^{\infty}}  +  \frac{1}{\tilde{\gamma} } \log \E \bigg[\exp\bigg(  \tilde{\gamma}  \int_{t}^{T}    \| B_s \eta_{s} \|^2\, d A_{s}  \bigg) \bigg| \,\mathcal{F}_{t}\bigg].
\]
We see immediately that if $b(\eta \cdot M)   >  \tilde{\gamma}$, then $\Psi$ is bounded. Applying this to the specific BSDE \eqref{BSDE} related to power utility maximization for $q < 0$ and using Lemma \ref{PropF} (ii) in the appendix, we obtain that the solution to \eqref{BSDE} has a bounded first component if $b\Big( \sqrt{\frac{q(q-\epsilon_{0})}{2 \epsilon_{0} }}\,\lambda \cdot M \Big) > 1- q + \epsilon_{0}$ or, equivalently,
\[b(\lambda \cdot M)   >  \frac{1}{2}\bigg(\frac{q^{2}(1-q)}{\epsilon_{0}} - q + 2q^{2}-q \epsilon_{0} \bigg) \quad  \textrm{for some } \epsilon_{0} > 0.\] Choosing the minimizing $\epsilon_{0} = \sqrt{-q (1-q )}$, the right hand side equals 
\begin{equation}
k_q:=q^2-\frac{q}{2}-q\sqrt{q^2-q}=\frac{1}{2}\,\Bigl(q-\sqrt{q^2-q}\Bigr)^{2}>0\label{defkq}
\end{equation}
so that $b(\lambda \cdot M)  > k_{q}$ implies existence of a solution with bounded first component.  

Contrary to this result, the specific example of a BMO martingale that does not yield a bounded solution to the BSDE in Subsection \ref{CountBound2} exhibits\[b(\lambda\cdot M)=-\frac{q}{2}<q^2-\frac{q}{2}-q\sqrt{q^2-q}=k_q,\]recalling that $q<0$ (and where the first equality can be shown using \eqref{bddcos} below). The following questions arise.
\begin{itemize}
\item Which boundedness properties do hold for solutions to the BSDE \eqref{BSDE2} for those $\lambda$ with $b(\lambda\cdot M)\in\bigl(-\frac{q}{2},\,k_q\bigr)$?
\item Can we use the critical exponent $b$ to characterize boundedness of solutions to the BSDE \eqref{BSDE2}?
\end{itemize}
We answer these questions in the next subsection by showing that the bound $k_q$ is indeed the minimal one which guarantees boundedness, hence cannot be improved. In doing so we provide a full description of the boundedness of solutions to the quadratic BSDE \eqref{BSDE2} with $\lambda\cdot M$ a BMO martingale in terms of the critical exponent $b$   in a general filtration. 

\subsection{Boundedness under Dynamic Exponential Moments}
We have seen that neither the BMO property of $ \lambda \cdot M$ nor an exponential moments condition guarantees the boundedness of a BSDE solution. While a counterexample showed that a simple combination of the two conditions does not suffice, we next see that a dynamic combination provides the required characterization. In particular, while the existence of all exponential moments of the mean-variance tradeoff is sufficient for the existence of a unique solution $(\hat{\Psi},\hat{Z},\hat{N})$ to \eqref{BSDE} with $\hat{\Psi}\in\mathfrak{E}$, the existence of all \emph{dynamic} exponential moments is sufficient for the existence of a unique solution with $\hat{\Psi}$ bounded, and in general this requirement cannot be dropped. We recall that by Lemma \ref{lemb} any requirement on the dynamic exponential moments may be cast in terms of a condition on the critical exponent $b$.

\begin{thm} \label{dem} Fix $p\in(0,1)$, i.e. $q < 0$, and define $k_{q}$ as in \eqref{defkq}. Then,
\begin{enumerate}
\item If $ \lambda \cdot M$ is a martingale with $b(\lambda\cdot M)>k_q$ then the solution pair $(\hat{X},\hat{Y})$ to the primal and dual problem exists and if $\hat{\Psi}$, $\hat{Z}$ and $\hat{N}$ are as in Theorem \ref{thmuniq}, then the triple $(\hat{\Psi},\hat{Z},\hat{N})$ is the unique solution to the BSDE \eqref{BSDE2} with $\hat{\Psi}$ \emph{bounded}.
\item For a one-dimensional Brownian motion $M = W$ and every $k  < k_{q}$, there exists a BMO martingale $ \lambda \cdot M $ with $b(\lambda\cdot M)>k$ such that the solutions to the primal and dual problem exist and the corresponding triple $(\hat{\Psi},\hat{Z},\hat{N}\equiv0)$ is a solution to the BSDE \eqref{BSDE2} with $\hat{\Psi}$ \emph{unbounded}.
\item For a one-dimensional Brownian motion $M = W$, there exists a BMO martingale $ \lambda \cdot M $ with $b(\lambda\cdot M)=k_q$ such that the solutions to the primal and dual problem exist and the corresponding triple $(\hat{\Psi},\hat{Z},\hat{N}\equiv0)$ is the unique solution to the BSDE \eqref{BSDE2} with $\hat{\Psi}$ \emph{bounded}.
\end{enumerate}
\end{thm}
We can summarize this result as follows: Item (i) gives a sufficient condition for boundedness of BSDE solutions in terms of dynamic exponential moments, which is less restrictive than a bound on the {BMO}$_2$ norm. Item (ii) shows that this condition is sharp in the sense that it cannot be improved. In particular, the critical exponent $b$ from \eqref{defb} characterizes the boundedness property of solutions to the BSDE \eqref{BSDE2} that stem from the utility maximization problem. Item (iii) gives information about the critical point $k_q$ in the interval $(k_q,+\infty).$ It yields that the converse of item (i) does not hold.

The following Figure \ref{fig_kq} provides a visualization of this discussion, it depicts the value $k_q$ as a function of $p$. Let us now discuss it briefly, fix $p\in(0,1)$ and assume that we are on the critical black line, i.e. we have a specific $\lambda\cdot M$ with $b(\lambda\cdot M)>k_q$. Note that the black line is included in the area that ensures boundedness because a finite dynamic exponential moment of order $k_{q}$ is equivalent to $b(\lambda\cdot M)>k_q$ by Lemma \ref{lemb}. Now choosing $\tilde{q}<0$ such that $b(\lambda\cdot M)>k_{\tilde{q}}>k_q$ we can derive the statement of Theorem \ref{dem} (i) for the corresponding $\tilde{p}>p$. However, $\tilde{p}$ depends on the specific choice of $\lambda$ and therefore, it is not possible to shift the whole black line uniformly for all processes $\lambda$.

\begin{figure}[ht!]
  \centering
  \includegraphics[height=5.5cm]{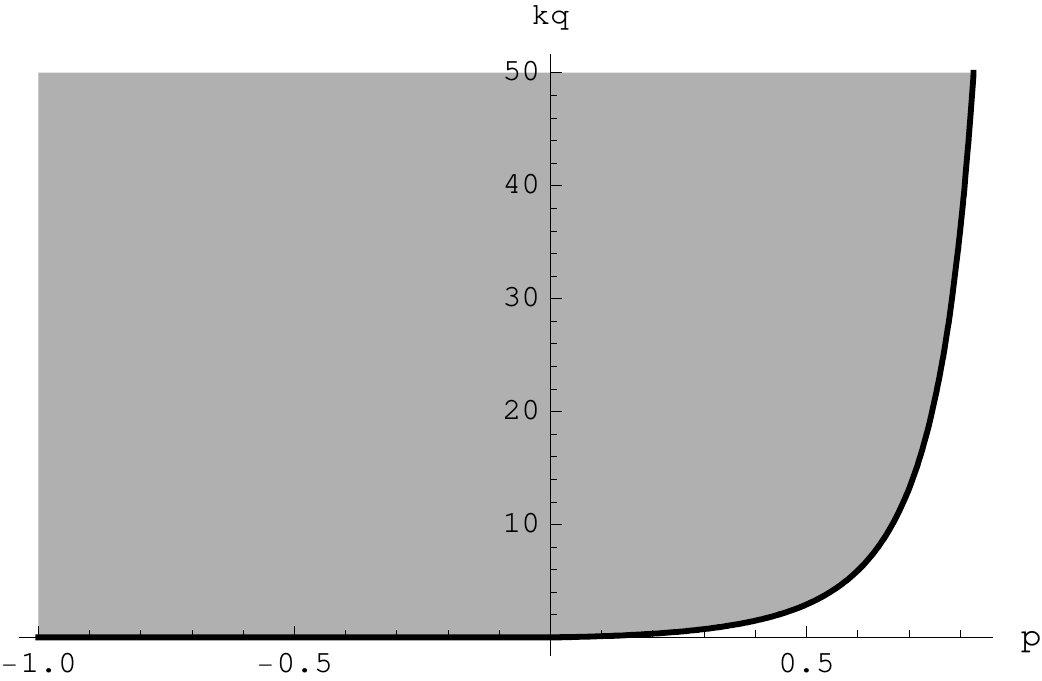}
  \caption{Dynamic exponential moments of $\lo\lambda\cdot M,\lambda\cdot M\ro$ sharply sufficient for the boundedness of $\hat{\Psi}$.}
  \label{fig_kq}
\end{figure}

Items (ii) and (iii) of the above theorem rely on the construction of a specific example which we provide in the following auxiliary lemma. 
\begin{lem}\label{lemcondex}
Let $W$ be a one-dimensional Brownian motion. Then, for every $b\in\mathbb{R}$, there exists a predictable process $\ti{\lambda}$ such that $\ti{\lambda}\cdot W$ is a BMO martingale and 
\begin{equation}\label{condex}
\sup_{\substack{ \tau \mathrm{\,stopping \,time}\\\mathrm{valued \,in\,[0,T]}}} \bigg\|  \E_{\ti{\mathbb{P}}}\biggl[ \exp \biggl( c^{2} \int_{\tau}^{T} \ti{\lambda}_{t}^{2} \,d t \biggr)\bigg| \,\mc{F}_{\tau} \biggr] \bigg\|_{L^{\infty}}
\begin{cases}
<+ \infty & \textrm{if }|c| < 1,  \\ =+\infty &  \textrm{if } |c| \geq 1,
\end{cases}
\end{equation}
where $\ti{\mathbb{P}}$ is the probability measure given by $\frac{d \ti{\mathbb{P}}}{d \mathbb{P}} := \mc{E} \bigl(- b \ti{\lambda}\cdot W\bigr)_{T}$.
\end{lem}
\begin{proof}
We proceed similarly to the example from Subsection \ref{CountBound2} and define for $t\in[0,T]$,
\begin{equation}\label{tillam}
\ti{\lambda}_t :=  \frac{\pi\alpha}{ \sqrt{8(T-t)}}\, \mathbf{1}_{\big]\!\!\big]T/2,\ti{\tau}\big]\!\!\big]}(t,\cdot),
\end{equation}
where $\alpha$ is as in the proof of Proposition \ref{ThmCountBound2} and where $\ti{\tau}$ is now the stopping time \[\ti{\tau} := \inf\Bigg\{ t >\frac{T}{2}\,\Bigg|\, \Bigg| \int_{T/2}^{t} \frac{1}{\sqrt{T-s}}\left(dW_{s}+ \frac{b \pi \alpha  }{\sqrt{8(T-s)}}\,ds \right) \!\Bigg|\geq1 \Bigg\},\] for which again $\mathbb{P}(T/2<\ti{\tau}<T)=1$. Then $\int_0^{{}^\cdot}  \ti{\lambda}_t \bigl( dW_t + b\ti{\lambda}_t \, dt \bigr)$ is bounded by $\frac{\pi}{\sqrt{8}}$. If $b<0$ we derive from \[\ti{\lambda}\cdot W=\int_0^\cdot  \ti{\lambda}_t \bigl( dW_t + b\ti{\lambda}_t \, dt \bigr)-b\int_0^\cdot \ti{\lambda}_t^2\,dt\geq-\frac{\pi}{\sqrt{8}},\] 
that the continuous local martingale $\ti{\lambda}\cdot W$ is bounded from below, hence a supermartingale. It then follows from the Optional Sampling Theorem, see Karatzas and Shreve \cite{KS91} Theorem 1.3.22, that for any stopping time $\tau$ valued in $[0,T]$,
\begin{align*}
\E\!\left[\int_\tau^T\ti{\lambda}_t^2\,dt\,\Bigg|\,\mc{F}_\tau\right]&=\frac{1}{b}\,\E\!\left[\int_\tau^T\ti{\lambda}_t(dW_t+b\ti{\lambda}_t\,dt)\,\Bigg|\,\mc{F}_\tau\right]-\frac{1}{b}\,\E\!\left[\int_\tau^T\ti{\lambda}_t\,dW_t\,\Bigg|\,\mc{F}_\tau\right]\leq\frac{\pi}{\sqrt{2}\,|b|}.
\end{align*}
In particular, $\ti{\lambda}\cdot W$ is a BMO martingale. A similar reasoning applies if $b>0$ and the claim is immediate for $b=0$. We hence may consider the measure $\ti{\mathbb{P}}$ given by $\frac{d \ti{\mathbb{P}}}{d \mathbb{P}} := \mc{E} \bigl(- b \ti{\lambda}\cdot W\bigr)_{T}$ under which $\ti{W}:= W +b\int_0^{{}^\cdot}\ti{\lambda}_t \, dt$ is a Brownian motion. 

Now, for a stopping time $\tau$ valued in $[T/2,T]$ and for $u\in \mathbb{R}$ and $v \in [0,T)$, we set
\begin{align*}
\ti{\tau}_{u,v}(\tau)  :&=  v+\inf\Biggl\{ t \geq 0 \,\Bigg|\, \Bigg| u+\int_{0}^{t} \frac{1}{\sqrt{T-s-v}}\,d\ti{W}_{\tau+s}  \Bigg|\geq1 \Biggr\}\\
&= \inf\Biggl\{ t \geq v \,\Bigg|\, \Bigg| u+\int_{v}^{t} \frac{1}{\sqrt{T-s}}\,d\ti{W}_{\tau+s-v}  \Bigg|\geq1 \Biggr\},
\end{align*}
where we extend the $\ti{\mathbb{P}}$-Brownian motion $\ti{W}$ to $[0,2T]$.

Let $|c|<1$. Since $\ti{\lambda}$ vanishes on $[0,T/2]$ and $\exp \bigl( c^{2} \int_{\tau}^{T} \ti{\lambda}_{t}^{2} \,d t \bigr)= 1$ on $\{\tau = T \}$, for the first assertion of \eqref{condex}, it is enough to consider stopping times $\tau$ valued in $(T/2,T)$. Using the $\mc{F}_{\tau}$-measurable random variable $U := \int^{\tau}_{T/2} \frac{1}{\sqrt{T-s}}\,d\ti{W}_{s}$ we have that $\ti{\tau}\leq{\ti{\tau}}_{{}_{U,\tau}}(\tau)$ a.s. Moreover, $\ti{\tau}_{u,v}(\tau)$ is $\ti{\mathbb{P}}$-independent of $\mc{F}_{\tau}$ since it is $\sigma\bigl(\ti{W}_{\tau+s}-\ti{W}_{\tau}, s\geq 0\bigr)$-measurable. We thus obtain
\begin{align}
\label{bddcos} \E_{\ti{\mathbb{P}}}\biggl[ \exp \biggl( c^{2} \int_{\tau}^{T} \ti{\lambda}_{t}^{2} \,d t \biggr)\bigg| \,\mc{F}_{\tau} \biggr]& \leq \E_{\ti{\mathbb{P}}}\biggl[ \exp \biggl( \frac{c^{2}\pi^{2}}{8} \int_{\tau}^{ {\ti{\tau}}_{{}_{U,\tau}}(\tau)} \frac{1}{T-t} \,d t \biggr) \bigg| \,\mc{F}_{\tau} \biggr]\\ \nonumber
& = \E_{\ti{\mathbb{P}}}\biggl[ \exp \biggl( \frac{c^{2}\pi^{2}}{8} \int_{v}^{\ti{\tau}_{u,v}(\tau)} \frac{1}{T-t} \,d t \biggr) \biggr]\Bigg|_{u = U,v=\tau}\\ \nonumber
& = \mathbf{1}_{\{ |U| \geq 1 \}} + \frac{\cos(c \pi U/ 2  )}{\cos (c \pi / 2 )}\,\mathbf{1}_{\{ |U|  < 1 \}} \leq \frac{1}{\cos (c \pi / 2 )}<+\infty,
\end{align}
where we applied \cite{Ka94} Lemma 1.3 in a similar way as in the proof of \cite{FdR10} Lemma A.1 and used that $\ti{\tau}_{u,v}(\tau)$ and $\ti{\tau}_{u,v}(0)$ have the same distribution under $\ti{\mathbb{P}}$. This gives an upper bound for \eqref{condex} in the case $|c| < 1$.

If $|c|\geq1$, we note that from $\ti{\tau}=\ti{\tau}_{0,T/2}(T/2)$ a.s. and the definition of $\alpha$, \begin{equation}\label{condex2}\E_{\ti{\mathbb{P}}}\biggl[ \exp \biggl( c^{2} \int_{T/2}^{T} \ti{\lambda}_{t}^{2} \,d t \biggr)\bigg| \,\mc{F}_{T/2} \biggr]\geq\E_{\ti{\mathbb{P}}}\biggl[ \exp \biggl( \frac{\pi^2\alpha^2}{8} \int_{T/2}^{\ti{\tau}} \frac{1}{T-t} \,d t \biggr)\bigg| \,\mc{F}_{T/2} \biggr]=\frac{1}{\cos(\pi\alpha/2)},\end{equation} which is \emph{unbounded} and this concludes the proof of Lemma \ref{lemcondex}.
\end{proof}
We are now ready to provide the proof of Theorem \ref{dem}.

\begin{proof}[Proof of Theorem \ref{dem}] For item (i) we proceed similarly to the proof of Lem\-ma \ref{Ass2341} by choosing the sharpest possible
version of H\"{o}lder's inequality in the sense that the condition
on the $\mathrm{BMO}_2$ norm of $\lambda\cdot M$ is the least
restrictive; this is how $k_q$ is selected. We set $\beta:=1-\frac{1}{q}\sqrt{q^2-q}>1$,
then with $\varrho:=\beta/(\beta-1)>1$, the dual number to $\beta$, we have that for any stopping time $\tau$ valued in $[0,T]$,
\begin{align}
\E\!\left[\Big(Y_T^{\lambda}\!\big/Y_\tau^{\lambda}\Big)^q\bigg|\,\mc{F}_\tau\right]&\leq \E\Bigg[\mc{E}(-\beta q\lambda\cdot
M)_{\tau,T}^{1/\beta}\exp\!\bigg(\frac{\varrho q}{2}(\beta q-1)\int_\tau^T\!\lambda_s^\tr\,d\lo M,M\ro_s\lambda_s\bigg)^{1/\varrho}\Bigg|\,\mc{F}_\tau\Bigg]\notag
\\
&\leq \E\Bigg[\exp\!\bigg(k_q\int_\tau^T\lambda_s^\tr\,d\lo
M,M\ro_s\lambda_s\bigg)\Bigg|\,\mc{F}_\tau\Bigg]^{1/\varrho} \leq C
\end{align}
for some constant $C$, where we used H\"{o}lder's inequality, the supermartingale
property of $\mc{E}(-\beta q\lambda\cdot M)$, the definition of
the constants and $b(\lambda\cdot M)>k_q$. Assumption \ref{ass2} holds because $Y^\lambda:=\mc{E}(-\lambda\cdot M)$ is a martingale by \cite{Ka94} Theorem 2.3. Moreover, using $x>0$ and $\tau\equiv0$ in the previous calculation, we obtain
\begin{align*}
0\leq
u(x)=\sup_{\substack{\nu\in\,\mc{A}}}\E\Bigl[U\big(X_T^{x,\nu}\big)\Bigr]&\leq
\E\!\left[\ti{U}\big(Y_T^{\lambda}\big)\right]+
\sup_{\substack{\nu\in\,\mc{A}}}\E\!\left[X_T^{x,\nu}
Y_T^{\lambda}\right]\leq
-\frac{1}{q}\,\E\!\left[\big(Y_T^{\lambda}\big)^q\right]+
x\\&\leq-\frac{1}{q}\,c_{\textit{\tiny{rH}},p}+ x<+\infty.
\end{align*}
For the uniqueness statement we assume that $\hat{X}$, $\hat{Y}$, $\hat{\Psi}$, $\hat{Z}$ and $\hat{N}$ are as in Theorem \ref{thmuniq}. Then $(\hat{\Psi},\hat{Z},\hat{N})$ is a solution to the BSDE \eqref{BSDE2} where the process $\hat{\Psi}$ is bounded. This is due to \cite{Nu109} Proposition 4.5. Conversely, if the triple $(\Psi,Z,N)$ is a solution to the BSDE \eqref{BSDE2} with $ \Psi $ bounded, we can identify it with $(\hat{\Psi},\hat{Z},\hat{N})$ by \cite{Nu209} Corollary 5.6 provided that the utility maximization is finite for some $\tilde{p} \in (p,1)$, which is a consequence of Lemma \ref{lemholder}.

For item (ii) observe that since 
\begin{equation*}
k<k_q:=q^2-\frac{q}{2}-q\sqrt{q^2-q},
\end{equation*}
there exists an $a>0$ such that 
\begin{equation}
\label{defnOfA}
k<q^2-\frac{q}{2}-q\sqrt{q^2-q-2a^2}.
\end{equation}
Choose such an $a$ and then set $b:=\frac{1}{a}\,(q-\sqrt{q^2-q-2a^2})<\frac{q}{a}<0$. We mention that the need for two parameters $a$ and $b$ stems from the fact that we have two conditions which must both be satisfied, the first concerns the finiteness of exponential moments and the second relates to the (un)boundedness of $\hat{\Psi}$. We then define $\ti{\lambda}$ and $\ti{\mathbb{P}}$ as in Lemma \ref{lemcondex} and observe that contrary to the previous examples the measure change is now part of the construction. Finally, we set $\lambda:= \frac1a  \ti{\lambda}$ and deduce for $t \in [0,T]$ that,
\begin{align*} 
\E\bigl[\mc{E}(-\lambda \cdot W)_{t,T}^{q}\,\big| \,\mc{F}_{t} \bigr] & =
\E_{\ti{\mb{P}}}\biggl[\exp\biggl(\! \biggl(b-\frac{q}{a}\biggr)\!\int_{t}^{T}\ti{\lambda}_{s} \,d \ti{W}_{s}+ \biggl(\frac{qb}{a}-\frac{q }{2a^{2}} - \frac{b^{2}}{2}\biggr)\!\int_{t}^{T}\ti{\lambda}_{s}^{2} \,d s  \biggr)\, \bigg|\, \mc{F}_{t}\biggr] \\ 
& \begin{cases}
 \leq  e^{  \frac{(q /a-b)\pi}{\sqrt{2}}} \E_{\ti{\mathbb{P}}}\Bigl[ \exp \Bigl(   \int_{t}^{T} \ti{\lambda}_{s}^{2} \,d s\Bigr)\Big| \,\mc{F}_t \Bigr],    \\ \geq e^{ \frac{(b-q /a)\pi}{\sqrt{2}}} \E_{\ti{\mathbb{P}}}\Bigl[ \exp \Bigl(  \int_{t}^{T} \ti{\lambda}_{s}^{2} \,d s \Bigr)\Big| \,\mc{F}_t \Bigr],  
 \end{cases}
\end{align*}
where we used the boundedness of $\ti{\lambda}\cdot\ti{W}$ and $ \frac{qb}{a}-\frac{q }{2a^{2}}- \frac{b^{2}}{2}= 1 $, together with $b<q/a$. By \eqref{condex2}, this shows that  $\E\bigl[\mc{E}(-\lambda \cdot W)_{T/2,T}^{q}\,\big|\, \mc{F}_{T/2} \bigr] $ is unbounded, whereas we have $ \E\bigl[\mc{E}(-\lambda \cdot W)_{T}^{q}\bigr] < + \infty$ since $\E_{\ti{\mathbb{P}}}\bigl[ \exp \bigl(\int_{0}^{T} \ti{\lambda}_{t}^{2} \,d t \bigr) \bigr]=2$, see the proof of Proposition \ref{ThmCountBound2}. Proposition \ref{ExplSolW} now yields the  existence of a solution $(\hat{\Psi},\hat{Z},\hat{N}\equiv0)$ and the identification with the primal and dual problems. The conclusion is that $\hat{\Psi}$ is unbounded. Moreover, using the boundedness of $\ti{\lambda}\cdot\ti{W}$ again, we have
\[\sup_{\tau} \bigg\|  \E\biggl[ \exp \biggl(  k \int_{\tau}^{T} \lambda_{t}^{2} \,d t \biggr)\bigg| \,\mc{F}_{\tau} \biggr] \bigg\|_{L^{\infty}} 
 \leq   e^{    \frac{|b| \pi}{\sqrt{2}}  }\,\sup_{\tau} \bigg\|  \E_{\ti{\mb{P}}}\biggl[ \exp \biggl(  \int_{\tau}^{T}\biggl(\frac{k}{ a^{2}} - \frac{b^{2}}{2} \biggr) \ti{\lambda}_{t}^{2} \,d t \biggr) \bigg| \,\mc{F}_{\tau}\biggr] \bigg\|_{L^{\infty}}\!\!.\]
This is finite by \eqref{condex} since the relation $\frac{k}{ a^{2}} - \frac{b^{2}}{2} < 1$  is equivalent to \[k  < a^{2}+\frac{a^{2}b^{2}}{2} = qab-\frac{q}{2}= q^{2}- \frac{q}{2} -q \sqrt{q^{2}-q-2a^{2}},\] which is inequality \eqref{defnOfA}.

The proof of item (iii) is similar to that of item (ii). We use the same definitions subject to the modification that now we must choose $a > 0$ and $b\in\mb{R}$ such that
\[\frac{qb}{a}-\frac{q }{2a^{2}}-\frac{b^{2}}{2} <  1\quad \textrm{ and } \quad \frac{k_{q}}{ a^{2}}-\frac{b^{2}}{2}= 1.\]
This choice ensures the existence of the optimizers and guarantees the boundedness of $\hat{\Psi}$, again thanks to Proposition \ref{ExplSolW} and \eqref{condex}. Note that now a dynamic exponential moment of order $k_q$ will not exist.

The above equation is satisfied for $b:= \sqrt{\frac{2k_{q}}{ a^{2}} -2}>0$ if $a^{2}  < k_{q} $, and then the inequality reads as $\frac{q}{a}\sqrt{\frac{2k_{q}}{ a^{2}} -2}-\frac{q }{2a^{2}} - \frac{k_{q}}{ a^{2}}    <  0 $. This last relation holds for any choice of $a \in \bigl(0,\sqrt{ k_{q}}\,\bigr)$ since we have  $k_{q} >-\frac{q}{2}> 0$.
\end{proof}

A consequence of Theorem \ref{dem} is the following result. 
\begin{cor} \label{cordem} \mbox{}
\begin{enumerate}
\item If $\lambda \cdot M $ is a martingale that satisfies $b( \lambda \cdot M) = +\infty$, then for all $p\in(0,1)$ the solution pair $(\hat{X},\hat{Y})$ to the primal and dual problem exists. If $\hat{\Psi}$, $\hat{Z}$ and $\hat{N}$ are as in Theorem \ref{thmuniq}, then the triple $(\hat{\Psi},\hat{Z},\hat{N})$ is the unique solution to the BSDE \eqref{BSDE2} with $\hat{\Psi}$ \emph{bounded}.
\item The converse statement, however, is not true. More precisely, if $\lambda \cdot M$ is a BMO martingale such that for all $p\in(0,1)$ the solutions to the primal and dual problem exist with $\hat{\Psi}$ bounded, the critical exponent \emph{need not} satisfy $b( \lambda \cdot M) = +\infty$.
\end{enumerate}
\end{cor}
\begin{proof} The first part is an immediate consequence of Theorem \ref{dem} (i). For the second part, we proceed similarly to the proof of its item (ii). Taking a one-dimensional Brownian motion $M = W$, we define $\lambda$ via \eqref{tillam} with $b = 1/2$ and $\lambda = \ti{\lambda}$. By construction,  $\int_0^{{}^\cdot} \lambda_t \bigl( dW_t + \frac{\lambda_t}{2}\, dt \bigr)$ is bounded by $\frac{\pi}{\sqrt{8}} $ so that for $q<0$
\[\sup_{\substack{ \tau \mathrm{\,stopping \,time}\\\mathrm{valued \,in\,[0,T]}}}\bigg\|  \E\biggl[ \exp \biggl( - q\int_{\tau}^{T}  \lambda_{t}\, dW_{t} - \frac{q}{2} \int_{\tau}^{T} \lambda_{t}^{2}  \, d t \biggr) \bigg| \,\mc{F}_{\tau}\biggr] \bigg\|_{L^{\infty}} \leq e^{\frac{-q\pi}{\sqrt{2}}}<+ \infty.  
\]
Hence, for all $p\in(0,1)$, the solutions to the primal and dual problem exist and the corresponding triple $(\hat{\Psi},\hat{Z},\hat{N}\equiv0)$ is the unique solution to the BSDE \eqref{BSDE2} with $\hat{\Psi}$ bounded. For the estimate on the process $\lo \lambda\cdot W,\lambda\cdot W\ro$ we have
\begin{multline*}
\sup_{\tau} \bigg\|  \E\biggl[ \exp \biggl(  k \int_{\tau}^{T} \lambda_{t}^{2} \,d t \biggr)\bigg| \,\mc{F}_{\tau} \biggr] \bigg\|_{L^{\infty}}
\geq  e^{  \frac{-\pi}{ \sqrt{8}} }\,\sup_{\tau} \bigg\|  \E_{\ti{\mathbb{P}}}\biggl[ \exp \biggl( \biggl(k- \frac{1}{8} \biggr)\!\int_{\tau}^{T}   \lambda_{t}^{2}\,d t \biggr) \bigg| \,\mc{F}_{\tau}\biggr] \bigg\|_{L^{\infty}}.
\end{multline*}
The right hand side is $+\infty$ when $k- \frac{1}{8}\geq 1$ by \eqref{condex}, this implies $b( \lambda \cdot W)\leq \frac98  <+\infty$ $\big($actually, $b( \lambda \cdot W)=\frac98\big)$ despite the fact that $\hat{\Psi}$ is bounded for arbitrary $p\in(0,1)$.
\end{proof}
\begin{rmk}
Corollary \ref{cordem} is based on the fact that $b( \lambda \cdot M) = +\infty$ is stronger than requiring that $\mc{E}(- \lambda \cdot M)$ satisfies the reverse H\"{o}lder inequality $R_{q}$ for all $q < 0$. However, there exists an equivalence between $b( \lambda \cdot M) = +\infty$ and a \emph{strengthened} reverse H\"{o}lder condition. It follows from \cite{DT10} Theorem 4.2 that $b( \lambda \cdot M) = +\infty$ holds if and only if 
for some (or equivalently, all) $\varrho \in [1,+\infty)$ and all $a \in \mb{C}$ there exists $c_{\varrho,a} >0$ such that
\[\E\!\left[ \bigg|\frac{\mc{E}(a\lambda \cdot M)_{\sigma}}{\mc{E}(a\lambda \cdot M)_{\tau}} \bigg|^\varrho\Bigg|\,\mc{F}_\tau\right]\leq c_{\varrho,a}\] for all stopping times $\tau \leq \sigma$ valued in $[0,T]$.
\end{rmk}
%%%%%%%%%%%%%%%%%%%%%%%%%%%%%%%%%%%%%%%%%%%%%%%%%%%%%%%%%%%%%
%Appendix
%%%%%%%%%%%%%%%%%%%%%%%%%%%%%%%%%%%%%%%%%%%%%%%%%%%%%%%%%%%%%

\appendix

\section{Quadratic Continuous Semimartingale BSDEs under Exponential Moments}\label{appendBSDE}
In this appendix we provide a short introduction to quadratic semimartingale BSDEs as described in \cite{Mo09,MW10}. In particular we show that all the assumptions of \cite{MW10} are satisfied and summarize the main results therein which are pertinent to the present study. Let us consider the BSDE \eqref{BSDE} on $[0,T]$,
\begin{multline}
d\Psi_t= Z_t^\tr\,dM_t + dN_t -\frac{1}{2}\,d\lo N,N\ro_t \\+\frac{q}{2}(Z_t+\lambda_t)^\tr\,d\lo M,M\ro_t(Z_t+\lambda_t)-\frac{1}{2}Z_t^\tr\,d\lo M,M\ro_tZ_t ,\quad \Psi_T=0.\label{BSDEF-}
\end{multline}
To prove existence and uniqueness one must first factor the process $\lo M,M\ro$. We set $A:=\arctan\!\left(\sum_{i=1}^d\lo M^i,M^i\ro\right)$ so that $A$ is bounded by $\pi/2$ and derive the absolute continuity of each of the $\lo M^i,M^j\ro$, $i,j\in\{1,\ldots,d\}$, with respect to $A$ from the Kunita-Watanabe inequality in order to get the existence of a predictable process $B$ valued in the space of $d\times d$ matrices such that $\lo M,M\ro=B^{\tr}B\cdot A$. The BSDE \eqref{BSDE} then becomes
\begin{equation}
d\Psi_t= Z_t^\tr\,dM_t + dN_t -F(t,\Psi_{t},Z_t)\,dA_t-\frac{1}{2}\,d\lo N,N\ro_t ,\quad \Psi_T=0,\label{BSDEF}
\end{equation}
where $F:[0,T]\times\Omega \times \mb{R}\times\mb{R}^d\to\mb{R}$ is a random predictable function, called the \emph{driver}, which in \eqref{BSDEF-} is given by
\[F(t,z)=-\frac{q}{2}(z+\lambda_t)^\tr B^{\tr}_tB_t(z+\lambda_t)+\frac{1}{2}z^\tr B^{\tr}_tB_tz=-\frac{q}{2}\|B_t(z+\lambda_t)\|^2+\frac{1}{2}\|B_tz\|^2.\]
Since the results in \cite{MW10} only depend on the \emph{boundedness} of $A$, in a $d$-dimensional Brownian setting we may set $A_t:=t$ for $t\in[0,T]$ and $B$ the identity matrix. 
\begin{defn}\label{DefnSol}
A \emph{solution} to the BSDE \eqref{BSDEF} is a triple $(\Psi,Z,N)$ of processes valued in $\R\times\R^d\times\R$
satisfying the equation \eqref{BSDEF} a.s. such that:
\begin{enumerate}
\item The function $t\mapsto \Psi_t$ is \emph{continuous} a.s.
\item The process $Z$ is predictable and satisfies $\int_0^TZ_t\,d\lo M,M\ro_tZ_t<+\infty$, a.s. hence is $M$-integrable.
\item The local martingale $N$ is continuous and orthogonal to each component of $M$, i.e. $\lo M^i,N \ro \equiv0$ for all $i=1,\ldots,d$.
\item We have that a.s.\[\int_0^T|F(t,\Psi_{t},Z_t)|\,dA_t+\lo N,N\ro_T<+\infty.\]
\end{enumerate}
The process $Z\cdot M+N$ is called the \emph{martingale part} of a solution.
\end{defn}
We collect some properties of the driver $F$ of \eqref{BSDEF-} in the following lemma whose proof is left to the reader.
\begin{lem}\label{PropF}We have that a.s.
\begin{enumerate}
\item The function $z\mapsto F(t,z)$ is continuously differentiable for all $t\in[0,T]$.
\item The function $F$ has quadratic growth in $z$, i.e. for arbitrary $\eps_0>0$ and all $(t,z)\in[0,T]\times\mb{R}^d$ we have \[|F(t,z)|\leq \frac{1}{2}\max\!\left(\frac{q(q-\eps_0)}{\eps_0},\frac{q}{1-q}\right)\|B_t\lambda_t\|^2+\frac{\gamma}{2}\,\|B_tz\|^2=:\alpha_t+\frac{\gamma}{2}\,\|B_tz\|^2,\] where $\gamma:=1-q+\eps_0>0.$
\item We have a local Lipschitz condition in $z$, i.e. for all $t \in [0,T]$ and $z_1,z_2\in\mb{R}^d$,
\begin{align*}|F(t,z_1)-F(t,z_2)| 
\leq \max\!\left(\frac{1-q}{2},|q|\right)\!\Bigl(\|B_t\lambda_t\| +\|B_tz_1\|+\|B_tz_2\|\Bigr)\|B_t(z_1-z_2)\|.
\end{align*}
\item The driver $F$ is convex in $z$ for all $t \in [0,T]$. More precisely, its Hessian with respect to $z$ is given by $D^2_zF(t,z)=(1-q)\,B^{\tr}_tB_t$, a positive semidefinite matrix.
\end{enumerate}
\end{lem}
Then, recalling Assumption \ref{ass_expmom} on the exponential moments of $\int_0^T\lambda_t^\tr\,d\lo M,M\ro_t\lambda_t$, we find \cite{MW10} Assumption 2.2 verified. The following theorem collects together \cite{MW10} Theorems 2.4, 2.5 and Corollary 4.2 (ii).
\begin{thm}\label{thmMW10} Suppose Assumption \ref{ass_expmom} holds.
\begin{enumerate}
\item There exists a solution $(\Psi,Z,N)$ to the BSDE \eqref{BSDEF} with $\Psi\in\mathfrak{E}$ and $Z\cdot M$ and $N$ two square-integrable martingales. 
\item If $(\Psi,Z,N)$ solves the BSDE \eqref{BSDEF} with $\Psi\in\mathfrak{E}$ then $Z\cdot M$ and $N$ are two square-integrable martingales.
\item If $(\Psi,Z,N)$ and $(\Psi',Z',N')$ both solve the BSDE \eqref{BSDEF} with $\Psi,\Psi'\in\mathfrak{E}$ then $\Psi$ and $\Psi'$, $Z\cdot M$ and $Z'\cdot M$ as well as $N$ and $N'$ are indistinguishable.
\end{enumerate}
\end{thm}
\bibliography{BSDEsUtMaxBMO}
\bibliographystyle{abbrv}
\end{document}